\documentclass[11pt,dvips,twoside,letterpaper]{article} 

\usepackage{times,fancyhdr}
\usepackage{graphicx}
\usepackage{geometry}
\usepackage{titlesec}

\titlelabel{\thetitle.\quad}                                

\usepackage[labelsep=period]{caption}                       

\makeatletter
\def\cleardoublepage{\clearpage\if@twoside \ifodd\c@page\else%
    \hbox{}%
    \thispagestyle{empty}%
    \newpage%
    \if@twocolumn\hbox{}\newpage\fi\fi\fi}
\makeatother

\usepackage{amsmath,amsthm,amssymb}
\usepackage{bm}

\usepackage{slashbox}
\usepackage{subfigure}
\usepackage{url}

\newtheorem{thm}{Theorem}[section]
\newtheorem{example}[thm]{Example}
\newtheorem{remark}[thm]{Remark}
\newtheorem{lemma}[thm]{Lemma}
\newtheorem{proposition}[thm]{Proposition}
\newtheorem{definition}[thm]{Definition}

\def\a{\alpha}
\def\ba{{\bm \alpha}}
\def\b{\beta}

\def\t{\tau}
\def\x{{\bf x}}
\def\l{{\bm\lambda}}
\def\LD{{_0D_t^\a}}
\def\LDa{{_aD_t^\a}}
\def\LDC{{_a^CD_t^\a}}
\def\RDC{{_t^CD_b^\a}}
\def\RD{{_tD_b^\a}}
\def\LDz{{_0D_t^{0.5}}}
\def\LI{{_aI_t^\a}}

\def\RI{{_tI_b^\a}}
\def\LD{{_aD_t^\a}}
\def\RD{{_tD_b^\a}}

\def\LHI{{_a\mathcal{I}_t^\a}}
\def\RHI{{_t\mathcal{I}_b^\a}}
\def\LHD{{_a\mathcal{D}_t^\a}}
\def\RHD{{_t\mathcal{D}_b^\a}}

\def\LHDHz{{_1\mathcal{D}_t^{0.5}}}

\def\GLa{{^{GL}_{\phantom{1}a}D_t^\a}}
\def\GLb{{^{GL}_{\phantom{1}t}D_b^\a}}
\def\sGLa{{^{sGL}_{\phantom{1}a}D_t^\a}}
\def\w{\left(\omega_k^\a\right)}
\def\wh{\left(\omega_k^{0.5}\right)}

\begin{document}
\title{
{\begin{flushleft}
\vskip 0.45in
{\normalsize\bfseries\textit{Chapter~V}}
\end{flushleft}
\vskip 0.45in
\bfseries\scshape Numerical Approximations to Fractional Problems of the Calculus of Variations and Optimal Control}}
\author{\bfseries\itshape Shakoor Pooseh\thanks{E-mail address: spooseh@ua.pt}~, 
Ricardo Almeida\thanks{E-mail address: ricardo.almeida@ua.pt}
~and Delfim F. M. Torres\thanks{E-mail address: delfim@ua.pt}\\
Center for Research and Development in Mathematics and Applications (CIDMA)\\
Department of Mathematics, University of Aveiro, 3810-193 Aveiro, Portugal}
\date{}
\maketitle
\thispagestyle{empty}
\setcounter{page}{1}


\thispagestyle{fancy}
\fancyhead{}
\fancyhead[L]{This is a preprint of a paper whose final and definite form appeared in:
Chapter V, Fractional Calculus in Analysis, Dynamics and Optimal Control (Editor: Jacky Cresson),
Series: Mathematics Research Developments, Nova Science Publishers, New York, 2014.\\
{\footnotesize \url{http://www.novapublishers.com/catalog/product_info.php?products_id=46851}}}
\fancyhead[R]{}
\fancyfoot{}
\renewcommand{\headrulewidth}{0pt}


\noindent \textbf{Keywords:} fractional calculus of variations, fractional optimal control,
numerical methods, direct methods, indirect methods

\vspace{.08in} \noindent \textbf{AMS Subject Classification:} 49K05, 49M25, 26A33


\section{Introduction}

A fractional problem of the calculus of variations and optimal control consists
in the study of an optimization problem in which the objective functional
or constraints depend on derivatives and integrals of arbitrary,
real or complex, orders. This is a generalization of the classical theory,
where derivatives and integrals can only appear in integer orders.


\subsection{Preliminaries}

Integer order derivatives and integrals have a unified meaning in the literature.
In contrast, there are several different approaches and definitions in fractional
calculus for derivatives and integrals of arbitrary order. The following
definitions and notations will be used throughout this chapter. See \cite{Kilbas}.

\begin{definition}[\bf Gamma function]
The Euler integral of the second kind
$$
\Gamma(z)=\int_0^\infty t^{z-1}e^{-t}dt, \qquad Re(z)>0,
$$
is called the gamma function.
\end{definition}
The gamma function has an important property, $\Gamma(z+1)=z\Gamma(z)$,
and hence $\Gamma(z)=(z-1)!$ for $z\in\mathbb{N}$, which allows
to extend the notion of factorial to real numbers.
Other properties of this special function can be found in \cite{Andrews}.


\pagestyle{fancy}
\fancyhead{}
\fancyhead[EC]{Shakoor Pooseh, Ricardo Almeida, Delfim F. M. Torres}
\fancyhead[EL,OR]{\thepage}
\fancyhead[OC]{Numerical Approximations to Fractional Problems \dots}
\fancyfoot{}
\renewcommand\headrulewidth{0.5pt}


\begin{definition}[\bf Mittag--Leffler function]
Let $\a > 0$. The function $E_\a$ defined by
$$
E_\a(z)=\sum_{j=0}^\infty \frac{z^j}{\Gamma(\a j+1)},
$$
whenever the series converges, is called the one parameter Mittag--Leffler function.
The two-parameter Mittag--Leffler function with parameters $\a, \beta >0$ is defined by
\begin{equation}
\label{Mittag}
E_{\a,\beta}(z)=\sum_{j=0}^\infty \frac{z^j}{\Gamma(\a j+\beta)}.
\end{equation}
\end{definition}

\begin{definition}[\bf Gr\"{u}nwald--Letnikov derivative]
Let $0<\a<1$ and $\binom{\a}{k}$ be the generalization of binomial coefficients to real numbers.
\begin{itemize}
\item The left Gr\"{u}nwald--Letnikov fractional derivative is defined as
\begin{equation}
\label{LGLdef}
\GLa x(t)=\lim_{h\rightarrow 0^+} \frac{1}{h^\a} \sum_{k=0}^\infty(-1)^k\binom{\a}{k}x(t-kh).
\end{equation}
\item The right Gr\"{u}nwald--Letnikov derivative is
\begin{equation}
\label{RGLdef}
\GLb x(t)=\lim_{h\rightarrow 0^+} \frac{1}{h^\a} \sum_{k=0}^\infty(-1)^k\binom{\a}{k}x(t+kh).
\end{equation}
\end{itemize}
\end{definition}

In the above mentioned definitions, $\binom{\a}{k}$
is the generalization of binomial coefficients
to real numbers, defined by
$$
\binom{\a}{k}=\frac{\Gamma(\a+1)}{\Gamma(k+1)\Gamma(\a-k+1)}.
$$
In this relation, $k$ and $\a$ can be any integer, real or complex number,
except that $\a\notin\{-1,-2,-3,\ldots\}$.

\begin{definition}[\bf Riemann--Liouville fractional integral]
Let $x(\cdot)$ be an integrable function in $[a,b]$ and $\a>0$.
\begin{itemize}
\item The left Riemann--Liouville fractional integral of order $\alpha$ is given by
$$
\LI x(t)=\frac{1}{\Gamma(\alpha)}
\int_a^t (t-\tau)^{\alpha-1}x(\tau)d\tau,
\quad t\in [a,b].
$$
\item The right Riemann--Liouville fractional integral
of order $\alpha$ is given by
$$
\RI x(t)=\frac{1}{\Gamma(\alpha)}\int_t^b (\tau-t)^{\a-1}x(\tau)d\tau,
\quad t\in [a,b].
$$
\end{itemize}
\end{definition}

\begin{definition}[\bf Riemann--Liouville fractional derivative]
Let $x(\cdot)$ be an absolutely continuous function in $[a,b]$, $x(\cdot)\in AC[a,b]$, and $0\leq \a<1$.
\begin{itemize}
\item The left Riemann--Liouville fractional derivative of order $\alpha$ is given by
\begin{equation*}
\LD x(t)=\frac{1}{\Gamma(1-\alpha)}\frac{d}{dt}\int_a^t (t-\t)^{-\alpha}x(\t)d\t,
\quad t\in [a,b].
\end{equation*}
\item
The right Riemann--Liouville fractional derivative of order $\a$ is given by
$$
\RD x(t)=\frac{1}{\Gamma(1-\a)}\left (-\frac{d}{dt}\right)
\int_t^b (\t-t)^{-\a}x(\t)d\t, \quad t\in [a,b].
$$
\end{itemize}
\end{definition}

Another type of fractional derivatives, introduced by Caputo,
is closely related to the Riemann--Liouville definitions.

\begin{definition}[\bf Caputo's fractional derivative]
For a function $x(\cdot)\in AC[a,b]$ with $0\leq \a<1$:
\begin{itemize}
\item The left Caputo fractional derivative of order $\alpha$ is given by
\begin{equation*}
\LDC x(t)=\frac{1}{\Gamma(1-\alpha)}\int_a^t (t-\t)^{-\alpha}\dot{x}(\t)d\t,
\quad t\in [a,b].
\end{equation*}
\item
The right Caputo fractional derivative of order $\a$ is given by
$$
\RDC x(t)=\frac{-1}{\Gamma(1-\a)} \int_t^b (\t-t)^{-\a}\dot{x}(\t)d\t,
\quad t\in [a,b].
$$
\end{itemize}
\end{definition}

\begin{definition}[\bf Hadamard fractional integral]
Let $x:[a,b]\to\mathbb R$.
\begin{itemize}
\item The left Hadamard fractional integral of order $\a>0$ is defined by
$$
\LHI x(t)=\frac{1}{\Gamma(\alpha)}
\int_a^t \left(\ln\frac{t}{\tau}\right)^{\alpha-1}\frac{x(\tau)}{\tau}d\tau,
\quad t\in ]a,b[.
$$
\item The right Hadamard fractional integral of order $\a>0$ is defined by
$$
\RHI x(t)=\frac{1}{\Gamma(\alpha)}
\int_t^b \left(\ln\frac{\tau}{t}\right)^{\a-1}\frac{x(\tau)}{\tau}d\tau,
\quad t\in ]a,b[.
$$
\end{itemize}
\end{definition}
When $\a=m$ is an integer, these fractional integrals are \textit{m-fold} integrals:
$$
{_a\mathcal{I}_t^m} x(t)=\int_a^t \frac{d\tau_1}{\tau_1}\int_a^{\tau_1}
\frac{d\tau_2}{\tau_2}\ldots \int_a^{\tau_{m-1}} \frac{x(\tau_m)}{\tau_m}d\tau_m,
$$
and
$$
{_t\mathcal{I}_b^m}x(t)=\int_t^b \frac{d\tau_1}{\tau_1}\int_{\tau_1}^b
\frac{d\tau_2}{\tau_2}\ldots \int_{\tau_{m-1}}^b \frac{x(\tau_m)}{\tau_m}d\tau_m.
$$

\begin{definition}[\bf Hadamard fractional derivative]
For $\a>0$ and $n=[\a]+1$,
\begin{itemize}
\item The left Hadamard fractional derivative of order $\a$  is defined by
$$
\LHD x(t)=\left(t\frac{d}{dt}\right)^n\frac{1}{\Gamma(n-\alpha)}
\int_a^t \left(\ln\frac{t}{\tau}\right)^{n-\alpha-1}\frac{x(\tau)}{\tau}d\tau,
\quad t\in ]a,b[.
$$
\item The right Hadamard fractional derivative of order $\a$  is defined by
$$
\RHD x(t)=\left(-t\frac{d}{dt}\right)^n\frac{1}{\Gamma(n-\alpha)}
\int_t^b \left(\ln\frac{\tau}{t}\right)^{n-\a-1}\frac{x(\tau)}{\tau}d\tau,
\quad t\in ]a,b[.
$$
\end{itemize}
\end{definition}

When $\a=m$ is an integer, we have
$$
{_a\mathcal{D}_t^m} x(t)=\left(t\frac{d}{dt}\right)^m x(t)
\mbox{ and } {_t\mathcal{D}_b^m}x(t)=\left(-t\frac{d}{dt}\right)^m x(t).
$$


\subsection{Fractional Calculus of Variations and Optimal Control}

Many generalizations to the classical calculus of variations
and optimal control have been made to extend the theory to cover
fractional variational and fractional optimal control problems.
A simple fractional variational problem, for example, consists
in finding a function $x(\cdot)$ that minimizes the functional
\begin{equation}
\label{Functional}
J[x(\cdot)]=\int_a^b L(t, x(t), \LD x(t))dt,
\end{equation}
where $\LD$ is the left Riemann--Liouville fractional derivative.
Typically, some boundary conditions are prescribed as $x(a)=x_a$
and/or $x(b)=x_b$. Classical techniques have been adopted to solve such problems.
The Euler--Lagrange equation for a Lagrangian of the form $L(t, x(t), \LD x(t))$
has been derived firstly in \cite{Riewe,Riewe1}. Many variants of necessary conditions
of optimality have been studied. A generalization of the problem
to include fractional integrals, i.e., $L=L(t,{_aI_t^{1-\a}}x(t),\LD x(t))$,
the transversality conditions of fractional variational problems and many other aspects
can be found in the literature of recent years. See \cite{Agrawal,AlD,Atan}
and references therein. Furthermore, it has been shown that a variational problem
with fractional derivatives can be reduced to a classical problem using
an approximation of the Riemann--Liouville fractional derivatives in terms
of a finite sum, where only derivatives of integer order are present \cite{Atan}.

On the other hand, fractional optimal control problems usually appear in the form of
\begin{align*}
& J[x(\cdot)]=\int_a^b L(t, x(t), u(t))dt\rightarrow \min\\
& s.t.~ \left\{
\begin{array}{l}
\LD x(t)=f(t,x(t),u(t))\\
x(a)=x_a,~x(b)=x_b,
\end{array}
\right.
\end{align*}
where an optimal control $u(\cdot)$ together with an optimal trajectory $x(\cdot)$
are required to follow a fractional dynamic and, at the same time,
optimize an objective functional. Again, classical techniques are generalized
to derive necessary optimality conditions. Euler--Lagrange equations have been introduced,
\textrm{e.g.}, in \cite{Agrawal1}. A Hamiltonian formalism for fractional optimal control
problems can be found in \cite{Ozlem} that exactly follows the same procedure
of the regular optimal control theory, i.e., those with only integer-order derivatives.

Due to the growing number of applications of fractional calculus in science and engineering
(see, \textrm{e.g.}, \cite{Das,Kai,Machado,Machado1}), numerical methods are being developed
to provide tools for solving such problems. Using the Gr\"{u}nwald--Letnikov approach,
it is convenient to approximate the fractional differentiation operator, $D^\a$,
by generalized finite differences. In \cite{Podlubny} some problems have been solved
by this approximation. In \cite{Kai1} a predictor-corrector method is
presented that converts an initial value problem into an equivalent Volterra integral equation,
while \cite{Agrawal2} shows the use of numerical methods to solve such integral equations.
A good survey on numerical methods for fractional differential equations can be found in \cite{Ford}.

A numerical scheme to solve fractional differential equations has been introduced
in \cite{Atan0,Atan2}, and \cite{Jelicic}, making an adaptation, uses this technique
to solve fractional optimal control problems. The scheme is based on an expansion formula
to approximate the Riemann--Liouville fractional derivative. The approximations transform
fractional derivatives into finite sums containing only derivatives of integer order.

In this chapter, we try to analyze problems for which an analytic solution is available.
This approach gives us the ability of measuring the accuracy of each method. To this end,
we need to measure how close we get to the exact solutions. We use the $L^2$-norm
and define the error function $E[x(\cdot),\tilde{x}(\cdot)]$ by
$$
E=\|x(\cdot)-\tilde{x}(\cdot)\|_2
=\left(\int_a^b [x(t)-\tilde{x}(t)]^2dt\right)^{\frac{1}{2}},
$$
where $x(\cdot)$ is defined on $[a,b]$.


\subsection{A General Formulation}

The appearance of fractional terms of different types, derivatives and integrals,
and the fact that there are several definitions for such operators,
makes it difficult to present a typical problem to represent all possibilities.
Nevertheless, one can consider the optimization of functionals of the form
\begin{equation}
\label{GenForm}
J[\x(\cdot)]=\int_a^b L(t, \x(t), D^\ba \x(t))dt
\end{equation}
that depends on a fractional derivative, $D^\ba$, in which $\x=(x_1,x_2,\ldots,x_n)$,
$\ba=(\a_1,\a_2,\ldots,\a_n)$ and $\a_i$, $i=1,2,\ldots,n$,
are arbitrary real positive numbers. The problem can be with
or without boundary conditions. Many settings of fractional variational
and optimal control problems can be transformed to the optimization
of \eqref{GenForm}. Constraints that usually appear in the calculus
of variations and are always present in optimal control problems
can be included in the functional using Lagrange multipliers.
More precisely, in presence of dynamic constraints as fractional
differential equations, we assume that it is possible to transform
such equations to a vector fractional differential equation of the form
\begin{equation*}
D^\ba \x(t)=f(t, \x(t)).
\end{equation*}
In this stage, we introduce a new variable
$\l=(\lambda_1,\lambda_2,\ldots,\lambda_n)$ and consider the optimization of
\begin{equation*}
J[\x(\cdot)]=\int_a^b \left[L(t, \x(t), D^\ba \x(t))
+\l(t)D^\ba \x(t)-\l(t)f(t, \x(t))\right]dt.
\end{equation*}

When the problem depends on fractional integrals, $I^\a$,
a new variable can be defined as $z(t)=I^\a x(t)$.
Recall that $D^\a I^\a x=x$ (see, \textrm{e.g.}, \cite{Kilbas}).
The equation
\begin{equation*}
D^\a z(t)=D^\a I^\a x(t)=x(t)
\end{equation*}
can be regarded as an extra constraint to be added to the original problem.
However, problems containing fractional integrals can be treated directly
to avoid the complexity of adding an extra variable to the original problem.
Interested readers are addressed to \cite{AlD,PATFracInt}.

Throughout this chapter, by a fractional variational problem,
we mainly consider the following one-variable problem with given boundary conditions:
\begin{align*}
& J[x(\cdot)]=\int_a^b L(t, x(t), D^\a x(t))dt\rightarrow \min\\
& s.t.~ \left\{
\begin{array}{l}
x(a)=x_a,\\
x(b)=x_b.
\end{array}
\right.
\end{align*}
In this setting $D^\a$ can be replaced by any fractional operator that is available in the literature,
say, Riemann--Liouville, Caputo, Gr\"{u}nwald--Letnikov, Hadamard and so forth. The inclusion
of constraints is done by Lagrange multipliers. The transition from this problem to the general one,
equation \eqref{GenForm}, is straightforward and is not discussed here.


\subsection{Solution Methods}

There are two main approaches to solve variational, including optimal control, problems.
On the one hand, there are direct methods. In a branch of direct methods, the problem
is discretized over a mesh on the interested time interval. Discrete values of the unknown
function on mesh points, finite differences for derivatives, and, finally, a quadrature
rule for the integral, are used. This procedure reduces the variational problem,
a continuous dynamic optimization problem, to static multi-variable optimization. Better accuracies
are achieved by refining the underlying mesh size. Another class of direct methods uses
function approximation through a linear combination of the elements of a certain basis,
\textrm{e.g.}, power series. The problem is then transformed into the determination
of the unknown coefficients. To get better results in this sense, is the matter
of using more adequate or higher order function approximations.

On the other hand, there are indirect methods that reduce a variational problem
to the solution of a differential equation by applying some necessary optimality conditions.
Euler--Lagrange equations and Pontryagin's maximum principle are used, in this context,
to make the transformation process. Once we solve the resulting differential equation,
an extremal for the original problem is reached. Therefore, to reach better results
using indirect methods, one has to employ powerful integrators. It is worth, however,
to mention here that numerical methods are usually used to solve practical problems.

These two methods have been generalized to cover fractional problems,
which is the essential subject of this chapter.


\section{Expansion Formulas to Approximate Fractional Derivatives}

This section is devoted to present two approximations for the Riemann--Liouville,
Caputo and Hadamard derivatives that are referred as fractional operators afterwards.
We introduce the expansions of fractional operators in terms
of infinite sums involving only integer order derivatives. These expansions
are then used to approximate fractional operators in problems
of the fractional calculus of variations and fractional optimal control.
In this way, one can transform such problems into classical variational
or optimal control problems. Hereafter, a suitable method, that can be found
in the classical literature, is employed to find an approximated solution
for the original fractional problem. Here we focus mainly on the left derivatives
and the details of extracting corresponding expansions for right derivatives
are given whenever it is needed to apply new techniques.


\subsection{Riemann--Liouville Derivative}

\subsubsection{ Approximation by a Sum of Integer Order Derivatives}

Recall the definition of the left Riemann--Liouville derivative for $\a\in(0,1)$,
\begin{equation}
\label{LeftD}
\LDa x(t)=\frac{1}{\Gamma(1-\a)}\frac{d}{dt}\int_a^t (t-\t)^{-\a}x(\t)d\t.
\end{equation}
The following theorem holds for any function $x(\cdot)$
that is analytic in an interval $(c,d)\supset[a,b]$.
See \cite{Atan} for a more detailed discussion
and \cite{Samko} for a different proof.

\begin{thm}
\label{ThmIntExp}
Let $(c,d)$, $-\infty<c<d<+\infty$, be an open interval in $\mathbb R$,
and $[a,b]\subset(c,d)$ be such that for each $t\in[a,b]$ the closed ball
$B_{b-a}(t)$, with center at $t$ and radius $b-a$, lies in $(c,d)$.
If $x(\cdot)$ is analytic in $(c,d)$, then
\begin{equation}
\label{ExpIntInf}
\LDa x(t)=\sum_{k=0}^{\infty}\frac{(-1)^{k-1}\a
x^{(k)}(t)}{k!(k-\a)\Gamma(1-\a)}(t-a)^{k-\a}.
\end{equation}
\end{thm}

\begin{proof}
Since $x(t)$ is analytic in $(c,d)$, and $B_{b-a}(t)\subset(c,d)$
for any $\t\in(a,t)$ with $t\in(a,b)$, the Taylor expansion
of $x(\t)$ at $t$ is a convergent power series, \textrm{i.e.},
$$
x(\t)=x(t-(t-\t))=\sum_{k=0}^{\infty}\frac{(-1)^k x^{(k)}(t)}{k!}(t-\t)^k,
$$
and then, by \eqref{LeftD},
\begin{equation}
\label{ExpIntErr}
\LDa x(t)=\frac{1}{\Gamma(1-\a)}\frac{d}{dt}
\int_a^t \left((t-\t)^{-\a}\sum_{k=0}^{\infty}\frac{(-1)^k x^{(k)}(t)}{k!}(t-\t)^k\right)d\t.
\end{equation}
Since $(t-\t)^{k-\a} x^{(k)}(t)$ is analytic, we can interchange integration with summation, so
\begin{eqnarray*}
\LDa x(t)&=& \frac{1}{\Gamma(1-\a)}\frac{d}{dt}\left(
\sum_{k=0}^{\infty}\frac{(-1)^k  x^{(k)}(t)}{k!}\int_a^t (t-\t)^{k-\a}d\t\right)\\
&=& \frac{1}{\Gamma(1-\a)}\frac{d}{dt}\sum_{k=0}^{\infty}\left(
\frac{(-1)^k x^{(k)}(t)}{k!(k+1-\a)}(t-a)^{k+1-\a}\right)\\
&=&\frac{1}{\Gamma(1-\a)}\sum_{k=0}^{\infty}\left(
\frac{(-1)^k x^{(k+1)}(t)}{k!(k+1-\a)}(t-a)^{k+1-\a}
+\frac{(-1)^k x^{(k)}(t)}{k!}(t-a)^{k-\a}\right)\\
&=&\frac{x(t)}{\Gamma(1-\a)}(t-a)^{-\a}\\
&&+\frac{1}{\Gamma(1-\a)}\sum_{k=1}^\infty \left(\frac{(-1)^{k-1}}{(k-\a)(k-1)!}
+\frac{(-1)^k}{k!}\right)x^{(k)}(t)(t-a)^{k-\a}.
\end{eqnarray*}
Observe that
\begin{eqnarray*}
\frac{(-1)^{k-1}}{(k-\a)(k-1)!}+\frac{(-1)^k}{k!}
&=&\frac{k(-1)^{k-1}+k(-1)^k-\a(-1)^k}{(k-\a)k!}\\
&=&\frac{(-1)^{k-1}\a}{(k-\a)k!},
\end{eqnarray*}
since for any $k=0,1,2,\ldots$ we have $k(-1)^{k-1}+k(-1)^{k}=0$.
Therefore, the expansion formula is reached as required.
\end{proof}

For numerical purposes, a finite number of terms in \eqref{ExpIntInf} is used and one has
\begin{equation}
\label{expanInt}
\LDa x(t)\approx\sum_{k=0}^{N}\frac{(-1)^{k-1}\a x^{(k)}(t)}{k!(k-\a)\Gamma(1-\a)}(t-a)^{k-\a}.
\end{equation}

\begin{remark}
With the same assumptions of Theorem~\ref{ThmIntExp}, we can expand $x(\t)$ at $t$,
$$
x(\t)=x(t+(\t-t))=\sum_{k=0}^{\infty}\frac{x^{(k)}(t)}{k!}(\t-t)^k,
$$
where $\t\in(t,b)$. Similar calculations result in the following
approximation for the right Riemann--Liouville derivative:
\begin{equation*}
\RD x(t)\approx\sum_{k=0}^{N}\frac{-\a x^{(k)}(t)}{k!(k-\a)\Gamma(1-\a)}(b-t)^{k-\a}.
\end{equation*}
\end{remark}

A proof for this expansion is available at \cite{Samko} that uses a similar relation
for fractional integrals. The proof discussed here, however,
allows to extract an error term for this expansion easily.


\subsubsection{Approximation Using Moments of a Function}
\label{secexpan}

By moments of a function, we have no physical or distributive sense in mind.
The naming comes from the fact that, during expansion, the terms of the form
\begin{equation}
\label{defVp}
V_p(t):=V_p(x(t))=(1-p)\int_a^t (\t-a)^{p-2}x(\t) d\t,
\quad p\in\mathbb{N}, \quad \t\geq a,
\end{equation}
resemble the formulas of central moments (\textrm{cf.} \cite{Atan2}).
We assume that $V_p(x(\cdot))$, $p\in \mathbb{N}$,
denotes the $(p-2)$th moment of a function $x(\cdot)\in AC^2[a,b]$.

The following lemma, that is given here without a proof, is the key relation
to extract an expansion formula for Riemann--Liouville derivatives.

\begin{lemma}[\bf \textrm{cf.} Lemma 2.12 of \cite{Kai}]
\label{LemRL}
Let $x(\cdot)\in AC[a,b]$ and $0 < \a < 1$.
Then the left Riemann--Liouville fractional derivative,
$\LDa x(\cdot)$, exists almost everywhere in $[a,b]$.
Moreover, $\LDa x(\cdot)\in L_p[a,b]$ for $1\leq p<\frac{1}{\a}$ and
\begin{equation}
\label{DecThm}
\LDa x(t)=\frac{1}{\Gamma(1-\a)}\left[ \frac{x(a)}{(t-a)^\a}
+\int_a^t (t-\tau)^{-\a}\dot{x}(\tau)d\tau \right], \qquad t\in (a,b).
\end{equation}
The same argument is valid for the right Riemann--Liouville derivative and
\begin{equation*}
\RD x(t)=\frac{1}{\Gamma(1-\a)}\left[ \frac{x(b)}{(b-t)^\a}
-\int_t^b (\tau-t)^{-\a}\dot{x}(\tau)d\tau \right], \qquad t\in (a,b).
\end{equation*}
\end{lemma}

\begin{thm}[\bf \textrm{cf.} \cite{Atan0}]
With the same assumptions of Lemma~\ref{LemRL},
the left Riemann--Liouville derivative can be expanded as
\begin{multline}
\label{expanMomInf}
\LDa x(t)=\frac{(t-a)^{-\a}}{\Gamma(1-\a)} x(t)+B(\a)(t-a)^{1-\a}\dot{x}(t)\\
-\sum_{p=2}^{\infty} \left[C(\a,p)(t-a)^{1-p-\a}V_p(t)
- \frac{\Gamma(p-1+\a)}{\Gamma(\a)\Gamma(1-\a)(p-1)!}(t-a)^{-\a} x(t)\right],
\end{multline}
where $V_p(t)$ is defined by \eqref{defVp} and
\begin{eqnarray*}
B(\a)&=&\frac{1}{\Gamma(2-\a)}\left(1
+\sum_{p=1}^{\infty}\frac{\Gamma(p-1+\a)}{\Gamma(\a-1)p!}\right),\\
C(\a,p)&=&\frac{1}{\Gamma(2-\a)\Gamma(\a-1)}\frac{\Gamma(p-1+\a)}{(p-1)!}.
\end{eqnarray*}
\end{thm}

\begin{proof}
Integration by parts on the right-hand-side of \eqref{DecThm} gives
\begin{multline}
\label{exp2}
\LD x(t)=\frac{x(a)}{\Gamma(1-\a)}(t-a)^{-\a}
+\frac{\dot{x}(a)}{\Gamma(2-\a)}(t-a)^{1-\a}\\
+\frac{1}{\Gamma(2-\a)}\int_a^t (t-\t)^{1-\a}\ddot{x}(\t)d\t.
\end{multline}
Since $(\t-a)\leq(t-a)$,
$$
(t-\t)^{1-\a}=(t-a)^{1-\a}\left(1-\frac{\t-a}{t-a}\right)^{1-\a}.
$$
Using the binomial theorem, we have
$$
\left(1-\frac{\t-a}{t-a}\right)^{1-\a}
=\sum_{p=0}^{\infty}\frac{\Gamma(p-1+\a)}{\Gamma(\a-1)p!}\left(\frac{\t-a}{t-a}\right)^p,
$$
in which the infinite series converges. Replacing for $(t-\t)^{1-\a}$ in \eqref{exp2} gives
\begin{eqnarray*}
\LDa x(t)&=&\frac{x(a)}{\Gamma(1-\a)}(t-a)^{-\a}+\frac{\dot{x}(a)}{\Gamma(2-\a)}(t-a)^{1-\a}\\
& &+\frac{(t-a)^{1-\a}}{\Gamma(2-\a)}\int_a^t\left(\sum_{p=0}^{\infty}\frac{\Gamma(p-1+\a)}{\Gamma(\a-1)p!}
\left(\frac{\t-a}{t-a}\right)^p\right)\ddot{x}(\t)d\t, \quad t>a.
\end{eqnarray*}
Interchanging the summation and integration operations is possible, and yields
\begin{eqnarray*}
\LDa x(t)&=&\frac{x(a)}{\Gamma(1-\a)}(t-a)^{-\a}
+\frac{\dot{x}(a)}{\Gamma(2-\a)}(t-a)^{1-\a}\\
& &+\frac{(t-a)^{1-\a}}{\Gamma(2-\a)}\sum_{p=0}^{\infty}
\frac{\Gamma(p-1+\a)}{\Gamma(\a-1)p!(t-a)^p}\int_a^t
(\t-a)^p\ddot{x}(\t)d\t, \quad t>a.
\end{eqnarray*}
Decomposing the infinite sum, integrating,
and doing another integration by parts, allow us to write
\begin{eqnarray*}
\LDa x(t)&=&\frac{x(a)}{\Gamma(1-\a)}(t-a)^{-\a}
+\frac{\dot{x}(a)}{\Gamma(2-\a)}(t-a)^{1-\a}
+\frac{(t-a)^{1-\a}}{\Gamma(2-\a)}\int_a^t \ddot{x}(\t)d\t \\
& &+\frac{(t-a)^{1-\a}}{\Gamma(2-\a)}\sum_{p=1}^{\infty}\frac{\gamma(\a,p)}{p!(t-a)^p}
\left[(t-a)^p\dot{x}(t)-p\int_a^t(\t-a)^{p-1}\dot{x}(\t)d\t\right]\\
&=&\frac{x(a)}{\Gamma(1-\a)}(t-a)^{-\a}+\frac{\dot{x}(t)}{\Gamma(2-\a)}(t-a)^{1-\a}
+\frac{(t-a)^{1-\a}}{\Gamma(2-\a)}\sum_{p=1}^{\infty}
\frac{\gamma(\a,p)}{p!}\dot{x}(t)\\
& &+\frac{(t-a)^{1-\a}}{\Gamma(2-\a)}\sum_{p=1}^{\infty}\frac{\gamma(\a,p)}{(p-1)!(t-a)^p}
\int_a^t(\t-a)^{p-1}\dot{x}(\t)d\t,
\end{eqnarray*}
where
$$
\gamma(\a,p)=\frac{\Gamma(p-1+\a)}{\Gamma(\a-1)}.
$$
Repeating this procedure again, and simplifying the results, ends the proof.
\end{proof}

The moments $V_p(t)$, $p=2,3,\ldots$, can be regarded as the solutions
to the following system of differential equations:
\begin{equation}
\label{sysVp}
\left\{
\begin{array}{l}
\dot{V}_p(t)=(1-p)(t-a)^{p-2}x(t)\\
V_p(a)=0, \qquad p=2,3,\ldots
\end{array}
\right.
\end{equation}

As before, a numerical approximation is achieved by taking only a finite number
of terms in the series \eqref{expanMomInf}. We approximate the fractional derivative as
\begin{equation}
\label{expanMom}
\LDa x(t)\approx A(t-a)^{-\a}x(t)+B(t-a)^{1-\a}\dot{x}(t)-\sum_{p=2}^N C(\a,p)(t-a)^{1-p-\a}V_p(t),
\end{equation}
where $A=A(\a,N)$ and $B=B(\a,N)$ are given by
\begin{eqnarray}
A(\a,N)&=&\frac{1}{\Gamma(1-\a)}\left(1
+\sum_{p=2}^N\frac{\Gamma(p-1+\a)}{\Gamma(\a)(p-1)!}\right),\label{A}\\
B(\a,N)&=&\frac{1}{\Gamma(2-\a)}\left(1
+\sum_{p=1}^N\frac{\Gamma(p-1+\a)}{\Gamma(\a-1)p!}\right)\label{B}.
\end{eqnarray}

\begin{remark}
\label{BnoB}
This expansion has been proposed in \cite{Djordjevic} and a simplification
has been made in \cite{Atan2}, which uses the fact that the infinite series
$\sum_{p=1}^{\infty}\frac{\Gamma(p-1+\a)}{\Gamma(\a-1)p!}$
tends to $-1$, and concludes that $B(\a)=0$, and thus
\begin{equation}
\label{expanAtan}
\LD x(t)\approx A(\a,N)t^{-\a}x(t)-\sum_{p=2}^N C(\a,p)t^{1-p-\a}V_p(t).
\end{equation}
In practice, however, we only use a finite number of terms in series. Therefore,
\begin{equation*}
1+\sum_{p=1}^N\frac{\Gamma(p-1+\a)}{\Gamma(\a-1)p!}\neq 0,
\end{equation*}
and we keep here the approximation in the form of equation \eqref{expanMom},
\cite{APTIndInt}. To be more precise, the values of $B(\a,N)$,
for different choices of $N$ and $\a$, are given in Table~\ref{tab}.
It shows that even for a large $N$, when $\a$ tends to one, $B(\a,N)$ cannot be ignored.
\end{remark}

\begin{table}[!ht]
\caption{$B(\a,N)$ for different values of  $\a$ and $N$.}
\centering
\begin{tabular}{|c|c c c c c c c|}
\hline
$N$         &    4   &    7   &   15   &   30   &   70   &  120   &  170   \\
\hline
$B(0.1,N)$  & 0.0310 & 0.0188 & 0.0095 & 0.0051 & 0.0024 & 0.0015 & 0.0011 \\
$B(0.3,N)$  & 0.1357 & 0.0928 & 0.0549 & 0.0339 & 0.0188 & 0.0129 & 0.0101 \\
$B(0.5,N)$  & 0.3085 & 0.2364 & 0.1630 & 0.1157 & 0.0760 & 0.0581 & 0.0488 \\
$B(0.7,N)$  & 0.5519 & 0.4717 & 0.3783 & 0.3083 & 0.2396 & 0.2040 & 0.1838 \\
$B(0.9,N)$  & 0.8470 & 0.8046 & 0.7481 & 0.6990 & 0.6428 & 0.6092 & 0.5884 \\
$B(0.99,N)$ & 0.9849 & 0.9799 & 0.9728 & 0.9662 & 0.9582 & 0.9531 & 0.9498 \\
\hline
\end{tabular}
\label{tab}
\end{table}

\begin{remark}
Similar computations give rise to an expansion formula for $\RD$,
the right Riemann--Liouville fractional derivative:
\begin{equation*}
\RD x(t)\approx A(b-t)^{-\a}x(t)-B(b-t)^{1-\a}\dot{x}(t)
-\sum_{p=2}^NC(\a,p)(b-t)^{1-p-\a}W_p(t),
\end{equation*}
where
$$
W_p(t)=(1-p)\int_t^b (b-\tau)^{p-2}x(\tau)d\tau.
$$
The coefficients $A=A(\a,N)$ and $B=B(\a,N)$ are the same as \eqref{A}
and \eqref{B} respectively, and $C(\a,p)$ is as before.
\end{remark}

\begin{remark}
As stated before, Caputo derivatives are closely related to those of Riemann--Liouville.
For any function, $x(\cdot)$, and for $\a\in(0,1)$ for which these two kind
of fractional derivatives, left and right, exist, we have
$$
\LDC x(t)=\LD x(t)-\frac{x(a)}{(t-a)^\a},
$$
and
$$
\RDC x(t)=\RD x(t)-\frac{x(b)}{(b-t)^\a}.
$$
Using these relations, we can easily construct approximation formulas
for the left and right Caputo fractional derivatives, \textrm{e.g.},
\begin{eqnarray*}
\LDC x(t)&\approx & A(\a,N)(t-a)^{-\a}x(t)+B(\a,N)(t-a)^{1-\a}\dot{x}(t)\\
&&-\sum_{p=2}^N C(\a,p)(t-a)^{1-p-\a}V_p(t)-\frac{x(a)}{(t-a)^\a}.
\end{eqnarray*}
\end{remark}


\subsubsection{ Examples}

To examine the approximations provided so far, we take some test functions,
and apply \eqref{expanInt} and \eqref{expanMom} to evaluate their fractional
derivatives. We compute $\LD x(t)$, with $\a=\frac{1}{2}$, for $x(t)=t^4$ and $x(t)=e^{2t}$.
The exact formulas for the fractional derivatives of polynomials are derived from
\begin{equation*}
\LDz(t^n)=\frac{\Gamma(n+1)}{\Gamma(n+1-0.5)}t^{n-0.5},
\end{equation*}
and for the exponential function one has
$$
\LDz(e^{\lambda t})=t^{-0.5}E_{1,1-0.5}(\lambda t),
$$
where $E_{\a,\beta}$ is the two parameter Mittag--Leffler function \eqref{Mittag}.

Figure~\ref{EvalInt} shows the results using approximation \eqref{expanInt}.
As we can see, the third approximations are reasonably accurate for both cases.
Indeed, for $x(t)=t^4$, the approximation with $N=4$ coincides with the
exact solution because the derivatives of order five and more vanish.
\begin{figure}[!ht]
  \begin{center}
    \subfigure[$\LDz(t^4)$]{\includegraphics[scale=0.46]{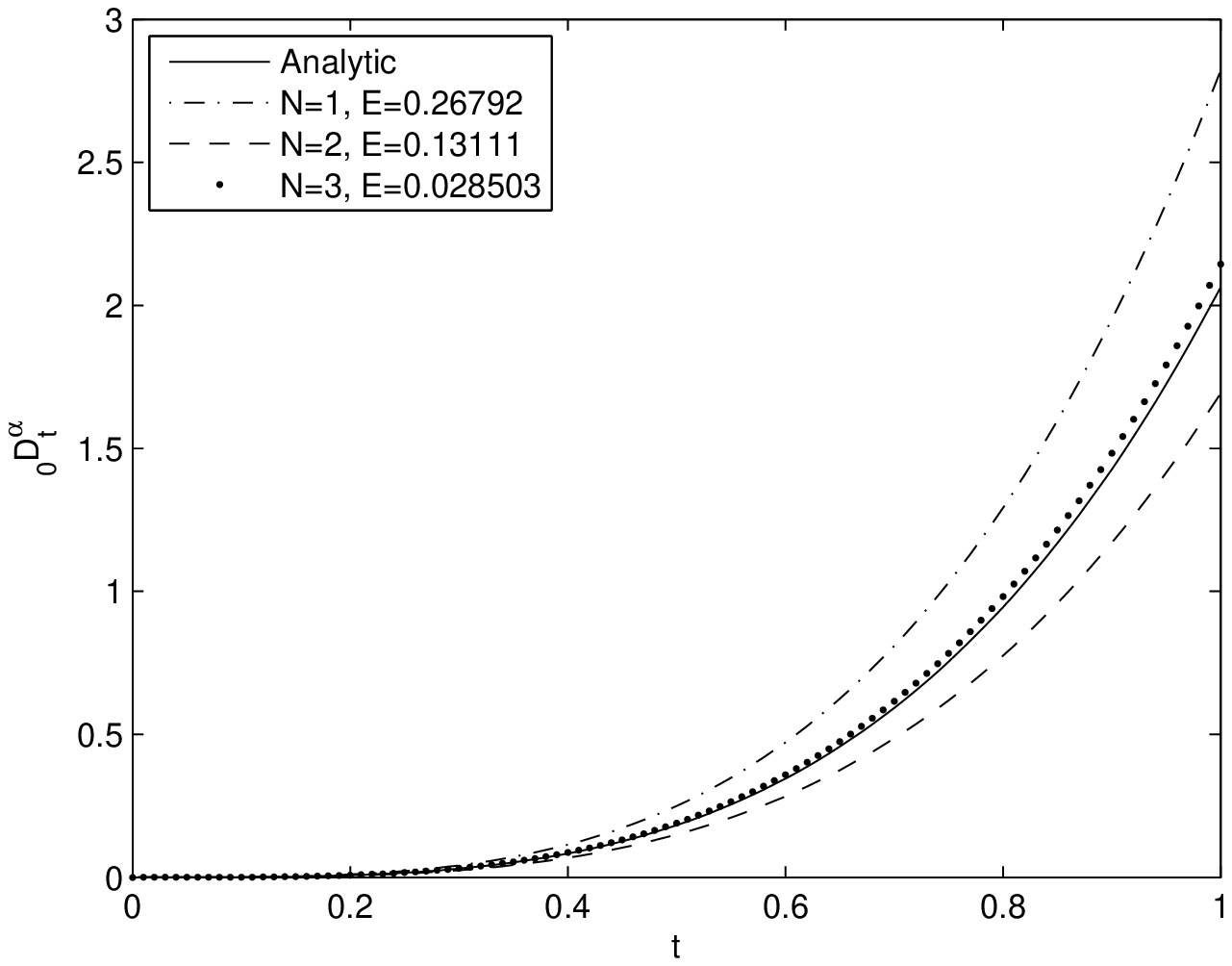}}
    \subfigure[$\LDz(e^{2t})$]{\includegraphics[scale=0.46]{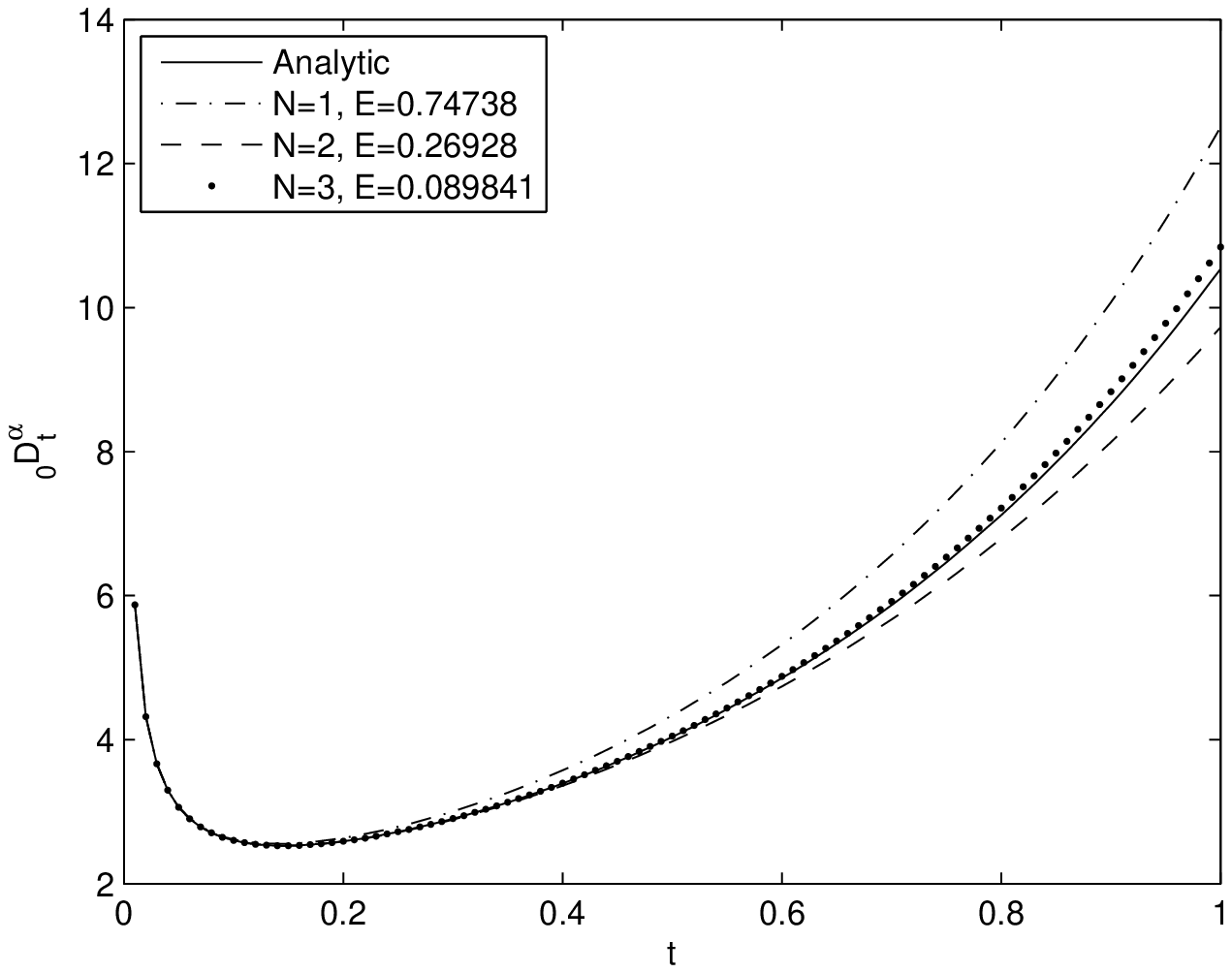}}
  \end{center}
  \caption{Analytic (solid line) versus numerical approximation \eqref{expanInt}.}
  \label{EvalInt}
\end{figure}

Now we use approximation \eqref{expanMom} to evaluate fractional derivatives
of the same test functions. In this case, for a given function $x(\cdot)$,
we can compute $V_p$ by definition, equation \eqref{defVp}. One can also integrate
the system \eqref{sysVp} analytically, if possible, or use any numerical integrator.
It is clearly seen in Figure~\ref{EvalMom} that one can get better results
by using larger values of $N$.
\begin{figure}[!ht]
  \begin{center}
    \subfigure[$\LDz(t^4)$]{\includegraphics[scale=0.46]{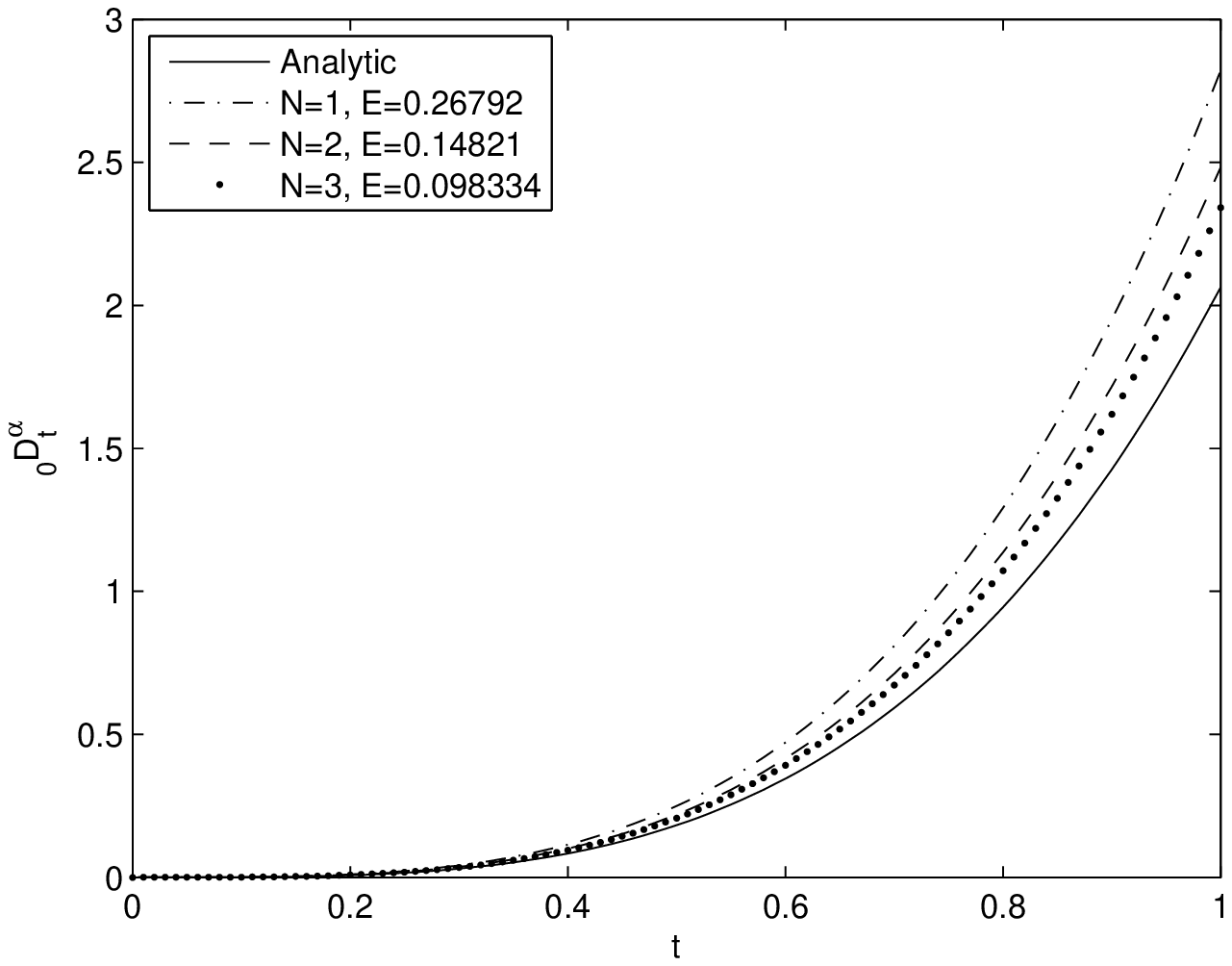}}
    \subfigure[$\LDz(e^{2t})$]{\includegraphics[scale=0.46]{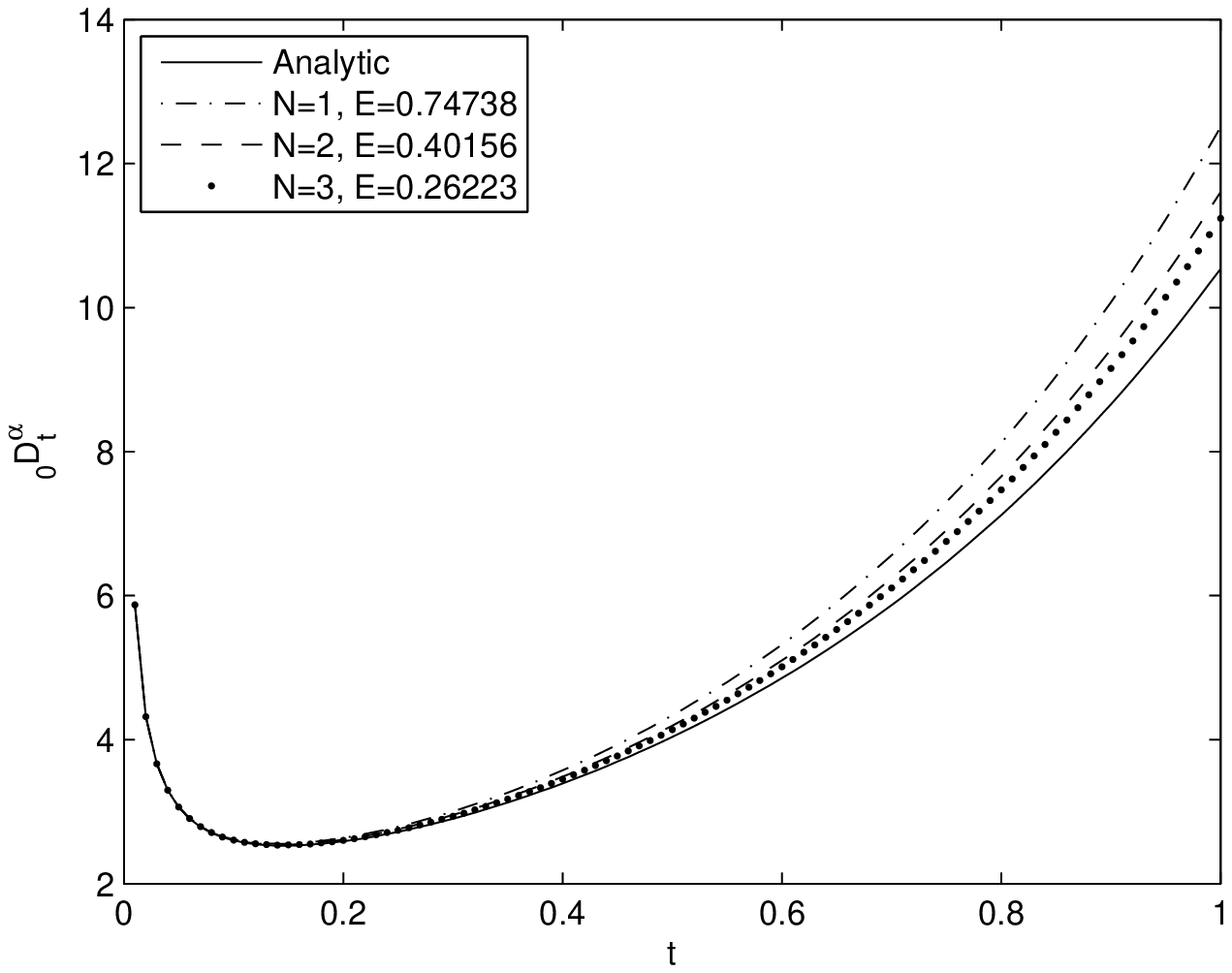}}
  \end{center}
  \caption{Analytic  (solid line) versus numerical approximation \eqref{expanMom}.}
  \label{EvalMom}
\end{figure}
Comparing Figures~\ref{EvalInt} and \ref{EvalMom}, we find out that the approximation
\eqref{expanInt} shows a faster convergence. Observe that both functions are analytic
and it is easy to compute higher-order derivatives.

\begin{remark}
A closer look to \eqref{expanInt} and \eqref{expanMom} reveals that in both cases
the approximations are not computable at $a$ and $b$ for the left and right fractional
derivatives, respectively. At these points we assume that it is possible
to extend them continuously to the closed interval $[a,b]$.
\end{remark}

Following Remark~\ref{BnoB}, we show here that neglecting the first derivative
in the expansion \eqref{expanMom} can cause a considerable loss of accuracy
in computation. Once again, we compute the fractional derivatives of $x(t)=t^4$
and $x(t)=e^{2t}$, but this time we use the approximation given by \eqref{expanAtan}.
Figure~\ref{compareEval} summarizes the results. Approximation \eqref{expanMom}
gives a more realistic approximation using quite small $N$, $3$ in this case.
\begin{figure}[!ht]
\begin{center}
\subfigure[$\LDz(t^4)$]{\includegraphics[scale=0.46]{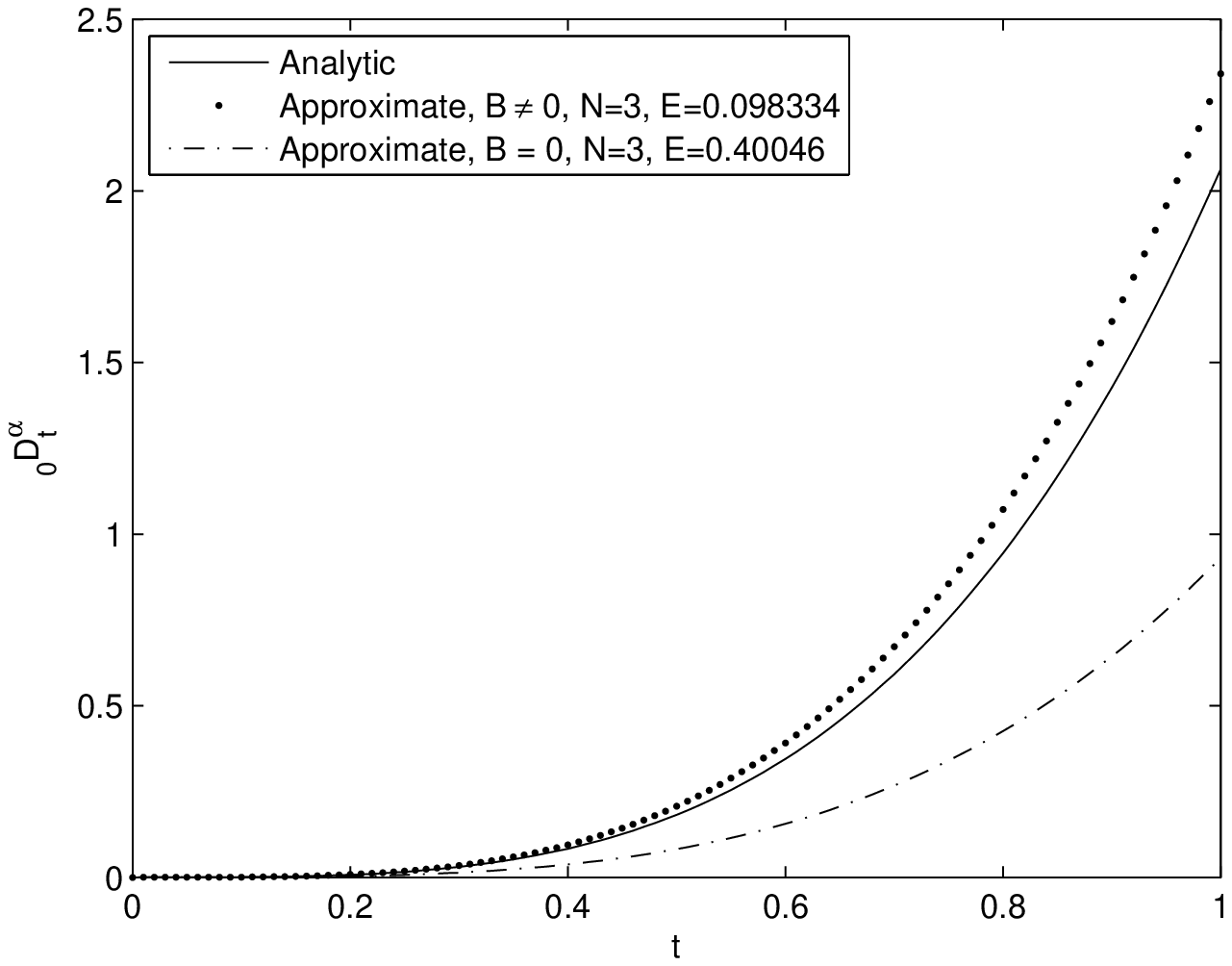}}
\subfigure[$\LDz(e^{2t})$]{\includegraphics[scale=0.46]{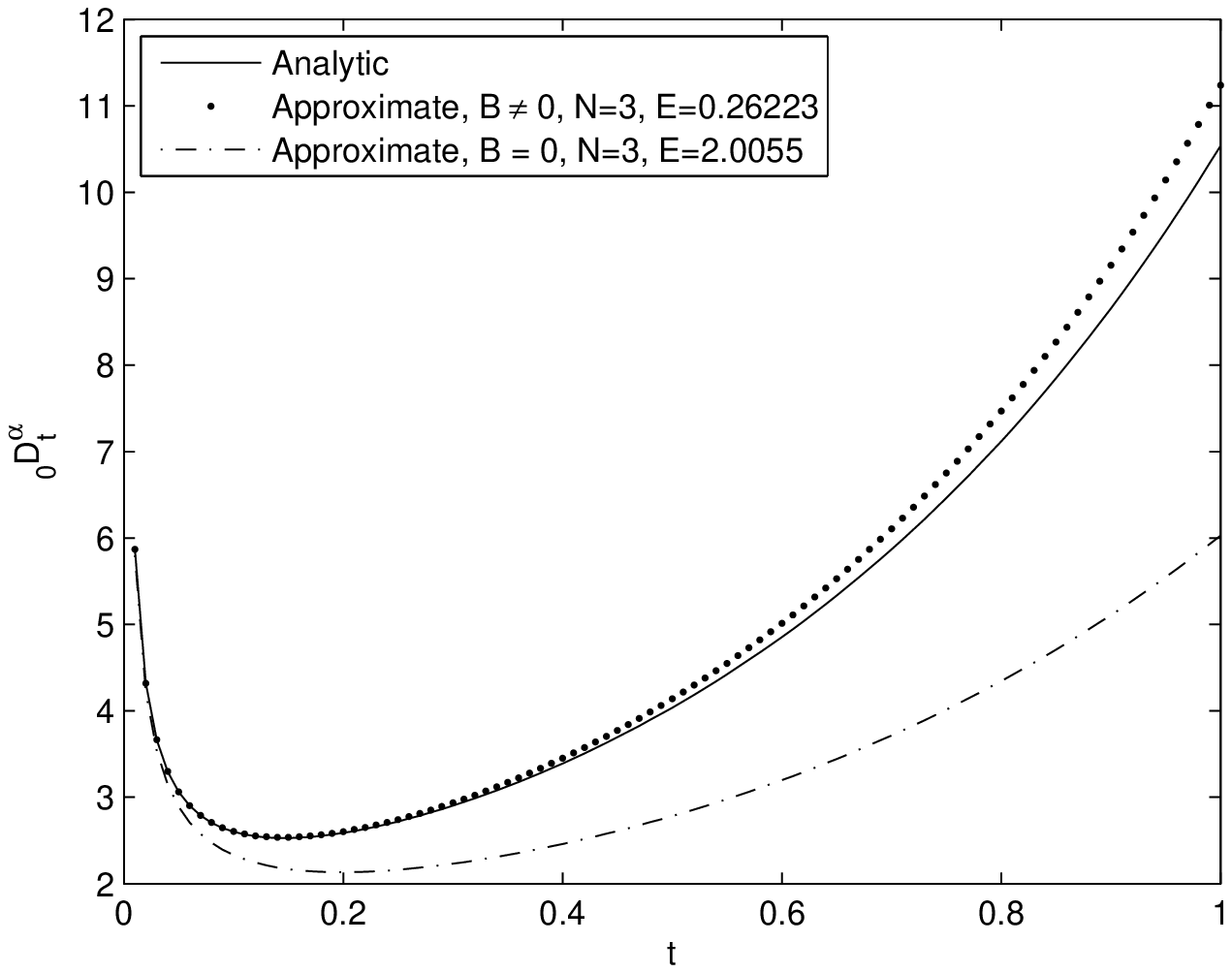}}
\end{center}
\caption{Comparison of approximation \eqref{expanMom}
and approximation \eqref{expanAtan} of \cite{Atan2}.}
\label{compareEval}
\end{figure}


\subsection{Hadamard Derivatives}

For Hadamard derivatives, the expansions can be obtained in a quiet similar way \cite{PATHad}.


\subsubsection{Approximation by a Sum of Integer Order Derivatives}

Assume that a function $x(\cdot)$ admits derivatives of any order,
then expansion formulas for the Hadamard fractional integrals
and derivatives of $x$, in terms of its integer-order derivatives,
are given in \cite[Theorem~17]{Butzer2}:
$$
{_0\mathcal{I}_t^\a}x(t)=\sum_{k=0}^\infty S(-\a,k)t^k x^{(k)}(t)
$$
and
$$
{_0\mathcal{D}_t^\a}x(t)=\sum_{k=0}^\infty S(\a,k)t^k x^{(k)}(t),
$$
where
$$
S(\a,k)=\frac{1}{k!}\sum_{j=1}^k(-1)^{k-j} {k \choose j} j^{\a}
$$
is the Stirling function.

As approximations, we truncate infinite sums at an appropriate
order $N$ and get the following formulas:
$$
{_0\mathcal{I}_t^\a}x(t)\approx\sum_{k=0}^N S(-\a,k)t^k x^{(k)}(t),
$$
and
$$
{_0\mathcal{D}_t^\a}x(t)\approx\sum_{k=0}^N S(\a,k)t^k x^{(k)}(t).
$$


\subsubsection{Approximation Using Moments of a Function}

The same idea of expanding Riemann--Liouville derivatives,
with slightly different techniques, is used to derive expansion
formulas for left and right Hadamard derivatives. The following
lemma is a basis for these new relations.

\begin{lemma}
Let $\a\in(0,1)$ and $x(\cdot)$ be an absolutely continuous function on $[a,b]$.
Then the Hadamard fractional derivatives may be expressed by
\begin{equation}
\label{LHDformula}
\LHD x(t)=\frac{x(a)}{\Gamma(1-\a)}\left(\ln\frac{t}{a}\right)^{-\a}
+\frac{1}{\Gamma(1-\a)}\int_a^t \left(\ln\frac{t}{\t}\right)^{-\a}\dot{x}(\t)d\t
\end{equation}
and
\begin{equation*}
\RHD x(t)=\frac{x(b)}{\Gamma(1-\a)}\left(\ln\frac{b}{t}\right)^{-\a}
-\frac{1}{\Gamma(1-\a)}\int_t^b \left(\ln\frac{\t}{t}\right)^{-\a}\dot{x}(\t)d\t.
\end{equation*}
\end{lemma}

A proof of this lemma, for an arbitrary $\a>0$,
can be found in \cite[Theorem~3.2]{Kilbas2}.

\begin{thm}
Let $0<a<b$ and $x:[a,b]\to\mathbb R$ be
an absolutely continuous function. Then
\begin{multline*}
\LHD x(t)=\frac{1}{\Gamma(1-\a)}\left(\ln\frac{t}{a}\right)^{-\a}x(t)
+B(\a)\left(\ln\frac{t}{a}\right)^{1-\a}t\dot{x}(t)\\
-\sum_{p=2}^\infty \left[C(\a,p)\left(\ln\frac{t}{a}\right)^{1-\a-p}V_p(t)
- \frac{\Gamma(p+\a - 1)}{\Gamma(\a)\Gamma(1-\a)(p-1)!}\left(
\ln\frac{t}{a}\right)^{-\a}x(t)\right]
\end{multline*}
with
\begin{eqnarray*}
B(\a)&=&\frac{1}{\Gamma(2-\a)}\left(1
+\sum_{p=1}^\infty\frac{\Gamma(p+\a-1)}{\Gamma(\a-1)p!}\right),\\
C(\a,p)&=&\frac{\Gamma(p+\a-1)}{\Gamma(-\a)\Gamma(1+\a)(p-1)!},\\
V_p(t)&=&(1-p)\int_a^t \left(\ln\frac{\t}{a}\right)^{p-2}\frac{x(\t)}{\t}d\t.
\end{eqnarray*}
\end{thm}

\begin{proof}
We rewrite \eqref{LHDformula} as
\begin{equation*}
\LHD x(t)=\frac{x(a)}{\Gamma(1-\a)}\left(\ln\frac{t}{a}\right)^{-\a}
+\frac{1}{\Gamma(1-\a)}\int_a^t \frac{1}{\t}\left(
\ln\frac{t}{\t}\right)^{-\a}\t\dot{x}(\t)d\t
\end{equation*}
and then integrating by parts gives
\begin{eqnarray*}
\LHD x(t)&=&\frac{x(a)}{\Gamma(1-\a)}\left(\ln\frac{t}{a}\right)^{-\a}
+\frac{a\dot{x}(a)}{\Gamma(2-\a)}\left(\ln\frac{t}{a}\right)^{1-\a}\\
&&+\frac{1}{\Gamma(2-\a)}\int_a^t \left(\ln\frac{t}{\t}\right)^{1-\a}[
\dot{x}(\t)+\t\ddot{x}(\t)]d\t.
\end{eqnarray*}
Now we use the following expansion for
$\left(\ln\frac{t}{\t}\right)^{1-\a}$, using the binomial theorem,
\begin{eqnarray*}
\left(\ln\frac{t}{\t}\right)^{1-\a}&=&\left(
\ln\frac{t}{a}\right)^{1-\a}\left(1-\frac{\ln\frac{\t}{a}}{\ln\frac{t}{a}}\right)^{1-\a}\\
&=&\left(\ln\frac{t}{a}\right)^{1-\a}\sum_{p=0}^\infty\frac{\Gamma(p-1+\a)}{\Gamma(\a-1)p!}
\cdot \frac{\left(\ln\frac{\t}{a}\right)^p}{\left(\ln\frac{t}{a}\right)^p}.
\end{eqnarray*}
This implies that
\begin{eqnarray*}
\LHD x(t)&=&\frac{x(a)}{\Gamma(1-\a)}\left(\ln\frac{t}{a}\right)^{-\a}
+\frac{a\dot{x}(a)}{\Gamma(2-\a)}\left(\ln\frac{t}{a}\right)^{1-\a}
+\frac{1}{\Gamma(2-\a)}\left(\ln\frac{t}{a}\right)^{1-\a}\\
&&\quad\times \sum_{p=0}^\infty\frac{\Gamma(p-1+\a)}{\Gamma(\a-1)p!}\left(
\ln\frac{t}{a}\right)^{-p}\int_a^t
\left(\ln\frac{\t}{a}\right)^{p}[\dot{x}(\t)+\t\ddot{x}(\t)]d\t.
\end{eqnarray*}
Extracting the first term of the infinite sum, simplifications and another integration
by parts using $u=\left(\ln\frac{\t}{a}\right)^{p}$,
$du=(p)\frac{1}{\t}\left(\ln\frac{\t}{a}\right)^{p-1}$
and $dv=[\dot{x}(\t)+\t\ddot{x}(\t)]d\t$, $v=\t\dot{x}(\t)$ yields
\begin{eqnarray*}
\LHD x(t)&=&\frac{x(a)}{\Gamma(1-\a)}\left(\ln\frac{t}{a}\right)^{-\a}
+B(\a)\left(\ln\frac{t}{a}\right)^{1-\a}t\dot{x}(t)
-\frac{1}{\Gamma(2-\a)}\left(\ln\frac{t}{a}\right)^{1-\a}\\
&&\quad \times \sum_{p=1}^\infty\frac{\Gamma(p-1+\a)}{\Gamma(\a-1)(p-1)!}\left(
\ln\frac{t}{a}\right)^{-p}\int_a^t \left(\ln\frac{\t}{a}\right)^{p-1}\dot{x}(\t)d\t.
\end{eqnarray*}
A final step of extracting the first term in the sum and integration by parts finishes the proof.
\end{proof}

For practical purposes, finite sums up to order $N$
are considered and the approximation becomes
\begin{eqnarray}
\label{HadAprx}
\LHD x(t)&\approx&A(\a,N)\left(\ln\frac{t}{a}\right)^{-\a}x(t)
+B(\a,N)\left(\ln\frac{t}{a}\right)^{1-\a}t\dot{x}(t)\nonumber\\
&&\quad+\sum_{p=2}^N C(\a,p)\left(\ln\frac{t}{a}\right)^{1-\a-p}V_p(t)
\end{eqnarray}
with
\begin{eqnarray*}
A(\a,N)&=&\frac{1}{\Gamma(1-\a)}\left(1
+\sum_{p=2}^N\frac{\Gamma(p+\a-1)}{\Gamma(\a)(p-1)!}\right),\\
B(\a,N)&=&\frac{1}{\Gamma(2-\a)}\left(1
+\sum_{p=1}^N\frac{\Gamma(p+\a-1)}{\Gamma(\a-1)p!}\right).
\end{eqnarray*}

\begin{remark}
The right Hadamard fractional derivative can be expanded in the same way.
This gives the following approximation:
\begin{eqnarray*}
\RHD x(t)&\approx& A(\a,N)\left(\ln\frac{b}{t}\right)^{-\a}x(t)
-B(\a,N)\left(\ln\frac{b}{t}\right)^{1-\a}t\dot{x}(t)\\
&&\quad-\sum_{p=2}^N C(\a,p)\left(\ln\frac{b}{t}\right)^{1-\a-p}W_p(t)
\end{eqnarray*}
with
\begin{equation*}
W_p(t)=(1-p)\int_t^b
\left(\ln\frac{b}{\tau}\right)^{p-2}\frac{x(\tau)}{\tau}d\tau.
\end{equation*}
\end{remark}


\subsubsection{Examples}

In this section we apply \eqref{HadAprx} to compute fractional derivatives,
of order $\a=\frac{1}{2}$, for $x(t)=t^4$ and $x(t)=\ln(t)$. The exact
Hadamard fractional derivative is available for $x(t)=t^4$ and we have
$$
\LHDHz (t^4)=\frac{\sqrt{\ln t}}{\Gamma(1.5)}.
$$
For $x(t)=\ln(t)$, only an approximation of the Hadamard fractional
derivative is found in the literature:
$$
\LHDHz \ln(t)\approx \frac{1}{\Gamma(0.5)\sqrt{\ln t}}
+\frac{0.5908179503}{\Gamma(0.5)}9t^9 \mbox{erf}(3\sqrt{\ln t}).
$$
The results of applying \eqref{HadAprx} to evaluate fractional
derivatives are depicted in Figure~\ref{HadTestFig}.

\begin{figure}[ht!]
\begin{center}
\subfigure[$\LHDHz(\ln t)$]{\includegraphics[scale=0.46]{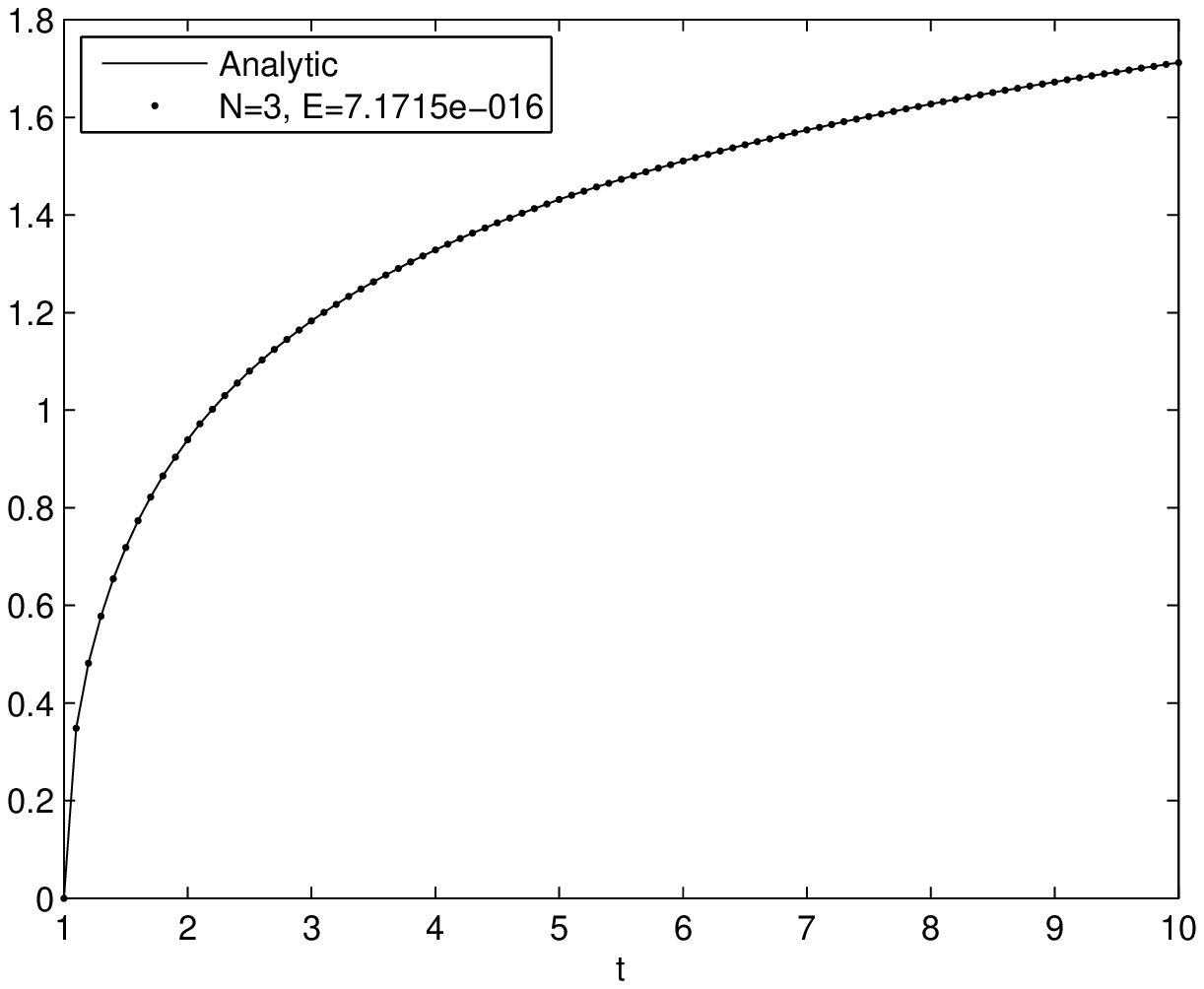}}
\subfigure[$\LHDHz(t^4)$]{\includegraphics[scale=0.46]{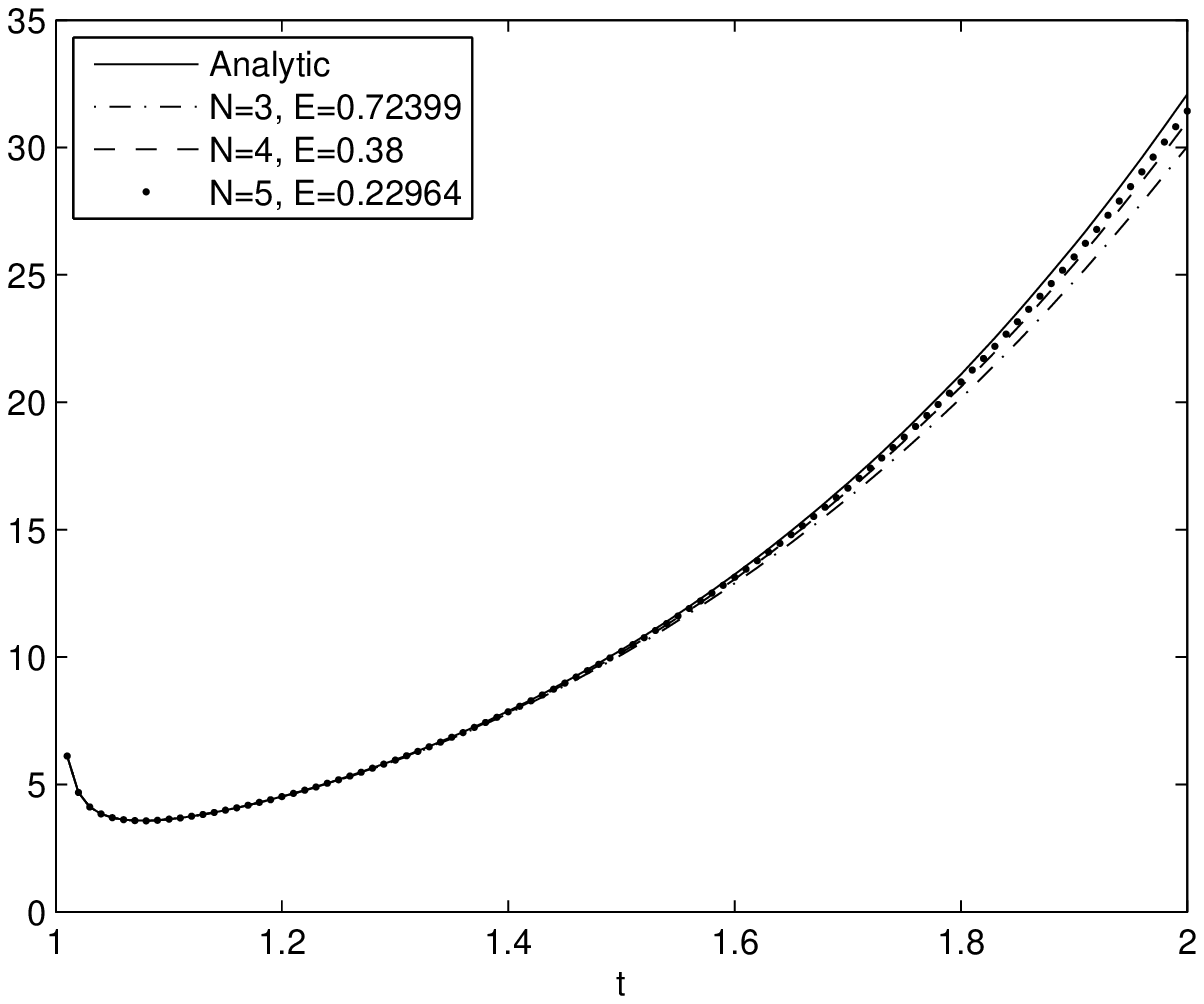}}
\end{center}
\caption{Analytic versus numerical approximation \eqref{HadAprx}.}
\label{HadTestFig}
\end{figure}


\subsubsection{Error Analysis}

When we approximate an infinite series by a finite sum, the choice
of the order of approximation is a key question. Having an estimate
knowledge of truncation errors, one can choose properly up to which
order the approximations should be made to suit the accuracy requirements.
In this section we study the errors of the approximations  presented so far.

Separation of an error term in \eqref{ExpIntErr} concludes in
\begin{multline}
\label{IntErr}
\LDa x(t)=\frac{1}{\Gamma(1-\a)}\frac{d}{dt}
\int_a^t \left((t-\t)^{-\a}\sum_{k=0}^{N}
\frac{(-1)^k x^{(k)}(t)}{k!}(t-\t)^k\right)d\t\\
+\frac{1}{\Gamma(1-\a)}\frac{d}{dt}\int_a^t
\left((t-\t)^{-\a}\sum_{k=N+1}^{\infty}
\frac{(-1)^k x^{(k)}(t)}{k!}(t-\t)^k\right)d\t.
\end{multline}
The first term in \eqref{IntErr} gives \eqref{expanInt} directly
and the second term is the error caused by truncation. The next
step is to give a local upper bound for this error, $E_{tr}(t)$.

The series
$$
\sum_{k=N+1}^{\infty}\frac{(-1)^k x^{(k)}(t)}{k!}(t-\t)^k,
\quad \t\in (a,t), \quad t\in (a,b),
$$
is the remainder of the Taylor expansion of $x(\t)$ and thus
bounded by $\left|\frac{M}{(N+1)!}(t-\t)^{N+1}\right|$ in which
$$
M=\displaystyle\max_{\t \in [a,t]}|x^{(N+1)}(\t)|.
$$
Then,
$$
E_{tr}(t)\leq \left|\frac{M}{\Gamma(1-\alpha)(N+1)!}\frac{d}{dt}
\int_a^t (t-\t)^{N+1-\alpha}d\t\right|
=\frac{M}{\Gamma(1-\alpha)(N+1)!}(t-a)^{N+1-\alpha}.
$$

In order to estimate a truncation error for approximation \eqref{expanMom},
the expansion procedure is carried out with separation
of $N$ terms in binomial expansion as
\begin{eqnarray}
\label{expanError}
\left(1-\frac{\t-a}{t-a}\right)^{1-\a}&=&\sum_{p=0}^{\infty}
\frac{\Gamma(p-1+\a)}{\Gamma(\a-1)p!}\left(\frac{\t-a}{t-a}\right)^p\nonumber\\
&=&\sum_{p=0}^{N}\frac{\Gamma(p-1+\a)}{\Gamma(\a-1)p!}\left(
\frac{\t-a}{t-a}\right)^p+R_N(\t),
\end{eqnarray}
where
$$
R_N(\t)=\sum_{p=N+1}^{\infty}\frac{\Gamma(p
-1+\a)}{\Gamma(\a-1)p!}\left(\frac{\t-a}{t-a}\right)^p.
$$
Substituting \eqref{expanError} into \eqref{exp2}, we get
\begin{eqnarray*}
\LD x(t)&=&\frac{x(a)}{\Gamma(1-\a)}(t-a)^{-\a}
+\frac{\dot{x}(a)}{\Gamma(2-\a)}(t-a)^{1-\a}\\
&& +\frac{(t-a)^{1-\a}}{\Gamma(2-\a)}\int_a^t
\left(\sum_{p=0}^{N}\frac{\Gamma(p-1+\a)}{\Gamma(\a-1)p!}\left(
\frac{\t-a}{t-a}\right)^p+R_N(\t)\right)\ddot{x}(\t)d\t\\
&=&\frac{x(a)}{\Gamma(1-\a)}(t-a)^{-\a}
+\frac{\dot{x}(a)}{\Gamma(2-\a)}(t-a)^{1-\a}\\
&& +\frac{(t-a)^{1-\a}}{\Gamma(2-\a)}\int_a^t
\left(\sum_{p=0}^{N}\frac{\Gamma(p-1+\a)}{\Gamma(\a-1)p!}
\left(\frac{\t-a}{t-a}\right)^p\right)\ddot{x}(\t)d\t\\
&& +\frac{(t-a)^{1-\a}}{\Gamma(2-\a)}\int_a^t R_N(\t)\ddot{x}(\t)d\t.
\end{eqnarray*}
At this point, we apply the techniques of \cite{Atan2} to the first
three terms with finite sums. Then, we receive \eqref{expanMom}
with an extra term of truncation error:
\begin{equation*}
E_{tr}(t)=\frac{(t-a)^{1-\a}}{\Gamma(2-\a)}\int_a^t R_N(\t)\ddot{x}(\t)d\t.
\end{equation*}
Since $0\leq\frac{\t-a}{t-a}\leq 1$ for $\t\in [a,t]$, one has
\begin{eqnarray*}
|R_N(\t)|&\leq & \sum_{p=N+1}^{\infty}\left|
\frac{\Gamma(p-1+\a)}{\Gamma(\a-1)p!}\right|
=\sum_{p=N+1}^{\infty}\left|\binom{1-\a}{p}\right|
\leq\sum_{p=N+1}^{\infty}\frac{\mathrm{e}^{(1-\a)^2+1-\a}}{p^{2-\a}} \\
&\leq&\int_{p=N}^{\infty}\frac{\mathrm{e}^{(1-\a)^2+1-\a}}{p^{2-\a}}dp
=\frac{\mathrm{e}^{(1-\a)^2+1-\a}}{(1-\a)N^{1-\a}}.
\end{eqnarray*}
Finally, assuming $L_2=\displaystyle\max_{\t \in [a,t]}\left|x^{(2)}(\t)\right|$,
we conclude that
\begin{equation*}
|E_{tr}(t)|\leq L_2 \frac{\mathrm{e}^{(1-\a)^2
+1-\a}}{\Gamma(2-\a)(1-\a)N^{1-\a}} (t-a)^{2-\a}.
\end{equation*}

\begin{remark}
Following similar techniques, one can extract an error bound
for the approximations of Hadamard derivatives. When
we consider finite sums in \eqref{HadAprx},
the error is bounded by
$$
\left| E_{tr}(t)\right|\leq L(t)\frac{e^{(1-\a)^2
+1-\a}}{\Gamma(2-\a)(1-\a)N^{1-\a}}\left(\ln\frac{t}{a}\right)^{1-\a}(t-a),
$$
where
$$
L(t)=\max_{\tau\in[a,t]}|\dot{x}(\tau)+\t\ddot{x}(\t)|.
$$
\end{remark}


\section{Direct Methods}

There are two main classes of direct methods in the classical calculus
of variations and optimal control. On the one hand, we specify
a discretization scheme by choosing a set of mesh points on the horizon
of interest, say $a=t_0,t_1,\ldots,t_n=b$ for $[a,b]$. Then we use some
approximations for derivatives in terms of unknown function values
at $t_i$ and, using an appropriate quadrature, the problem is transformed
to a finite dimensional optimization problem. This method is known
as Euler's method in the literature \cite{Elsgolts}. Regarding
Figure~\ref{EulerMethod}, the solid line is the function that we are looking for,
nevertheless, the method gives the polygonal dashed line as an approximate solution.
\begin{figure}[!htp]
\begin{center}
\setlength\fboxrule{0pt}
\fbox{\includegraphics[scale=1]{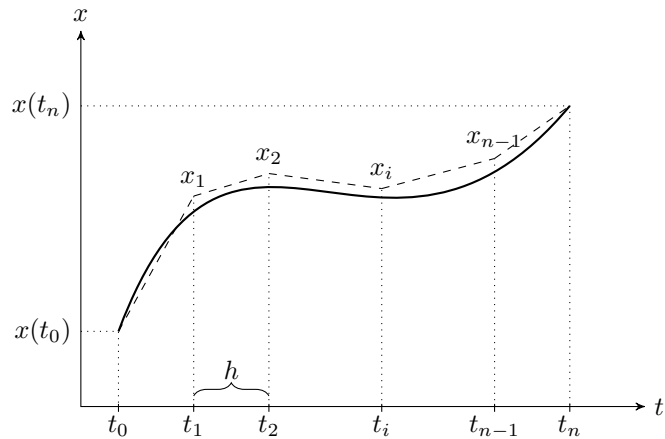}}
\end{center}
\caption{Euler's finite differences method.}\label{EulerMethod}
\end{figure}

On the other hand, there is the Ritz method, that has an extension to functionals
of several independent variables which is called Kantorovich's method. We assume
that the admissible  functions can be expanded in some kind of series,
\textrm{e.g.}, power or Fourier's series, of the form
$$
x(t)=\sum_{k=0}^\infty a_k \phi_k(t).
$$
Using a finite number of terms in the sum as an approximation, and some sort
of quadrature again, the original problem can be transformed to an equivalent
optimization problem for $a_k$, $k=0,1,\ldots,n$.

In the presence of fractional operators, the same ideas are applied to discretize
a problem. Many works can be found in the literature that use different types
of basis functions to establish Ritz-like methods for fractional calculus
of variations and optimal control.


\subsection{Euler-like Methods}

The Euler method in the classical theory of the calculus of variations uses finite
differences approximations for derivatives and is referred also as the method
of finite differences. The basic idea of this method is that instead
of considering the values of a functional
\begin{equation*}
J[x(\cdot)]=\int_a^b L(t, x(t), \dot{x}(t))dt
\end{equation*}
with boundary conditions $x(a)=x_a$ and $ x(b)=x_b$, on arbitrary admissible curves,
we only track the values at an $n+1$ grid points, $t_i$, $i=0,\ldots,n$,
of the interested time interval \cite{fda12}. The functional $J[x(\cdot)]$
is then transformed into a function $\Psi(x(t_1),x(t_2),\ldots,x(t_{n-1}))$
of the values of unknown function on mesh points. Assuming
$h=t_{i}-t_{i-1}$, $x(t_i)=x_i$
and $\dot{x}_i\approx \frac{x_{i}-x_{i-1}}{h}$, one has
\begin{eqnarray*}
J[x(\cdot)]&\approx&\Psi(x_1,x_2,\ldots,x_{n-1})
=h\sum_{i=1}^{n}L\left(t_i, x_i,\frac{x_{i}-x_{i-1}}{h}\right),\\
&&x_0=x_a,\quad x_n=x_b.
\end{eqnarray*}
The desired values of $x_i$, $i=1,\ldots,n-1$, are the extremum
of the multi-variable function $\Psi$ which is the solution to the system
$$
\frac{\partial \Psi}{\partial x_i}=0,\quad i=1,\ldots,n-1.
$$

The fact that only two terms in the sum, $(i-1)$th and $i$th, depend
on $x_i$, makes it rather easy to find the extremum of $\Psi$
solving a system of algebraic equations. For each $n$, we obtain
a polygonal line which is an approximate solution of the original problem.
It has been shown that passing to the limit as $h\rightarrow 0$,
the linear system corresponding to finding the extremum of $\Psi$
is equivalent to the Euler--Lagrange equation of the problem.


\subsubsection{Finite Differences for Fractional Derivatives}

In classical theory, given a derivative of a certain order, $x^{(n)}$,
there is a finite difference approximation of the form
\begin{equation*}
x^{(n)}(t)=\lim_{h\rightarrow 0^+} \frac{1}{h^n}
\sum_{k=0}^n(-1)^k\binom{n}{k}x(t-kh),
\end{equation*}
where $\binom{n}{k}$  is the binomial coefficient and
\begin{equation*}
\binom{n}{k}=\frac{n(n-1)(n-2)\cdots (n-k+1)}{k!},\quad n,k \in \mathbb{N}.
\end{equation*}
The Gr\"{u}nwald--Letnikov definition of fractional derivative is a generalization
of this formula to derivatives of arbitrary order.

The series in \eqref{LGLdef} and \eqref{RGLdef}, the Gr\"{u}nwald--Letnikov definitions,
converge absolutely and uniformly if $x(\cdot)$ is bounded. The infinite sums,
backward differences for the left and forward differences for the right derivative
in the Gr\"{u}nwald--Letnikov definitions for fractional derivatives, reveals
that the arbitrary order derivative of a function at a time $t$ depends
on all values of that function in $(-\infty,t]$ and $[t,\infty)$, for left
and right derivatives respectively. This is due to the non-local
property of fractional derivatives.

\begin{remark}
Equations \eqref{LGLdef} and \eqref{RGLdef} need to be consistent in closed
time intervals and we need the values of $x(t)$ outside the interval $[a,b]$.
To overcome this difficulty, we can take
\begin{equation*}
x^*(t)=\left\{
\begin{array}{ll}
x(t)& t\in [a,b],\\
0 & t\notin [a,b].
\end{array}\right.
\end{equation*}
Then we assume $\GLa x(t)=\GLa x^*(t)$
and $\GLb x(t)=\GLb x^*(t)$ for $t\in [a,b]$.
\end{remark}

This definition coincides with Riemann--Liouville and Caputo derivatives.
The latter is believed to be more applicable in practical fields
such as engineering and physics.

\begin{proposition}[See \cite{Podlubny}]
Let $0< \a<n$, $n\in \mathbb{N}$ and $x(\cdot)\in C^{n-1}[a,b]$.
Suppose also that $x^{(n)}(\cdot)$ is integrable on $[a,b]$. Then,
for every $\a$, the  Riemann--Liouville derivative exists and coincides
with the Gr\"{u}nwald--Letnikov derivative and the following holds:
\begin{eqnarray*}
\LDa x(t)&=&\sum_{i=0}^{n-1}\frac{x^{(i)}(a)(t-a)^{i-\a}}{\Gamma(1+i-\a)}
+ \frac{1}{\Gamma(n-\a)}\int_a^t(t-\t)^{n-1-\a}x^{(n)}(\t)d\t\\
&=&\GLa x(t).
\end{eqnarray*}
\end{proposition}

\begin{remark}
For numerical purposes we need a finite series in \eqref{LGLdef}. Given
a grid on $[a,b]$ as $a=t_0,t_1,\ldots,t_n=b$, where $t_i=t_0+ih$
for some $h>0$, we approximate the left Riemann--Liouville  derivative as
\begin{equation}
\label{GLApprx}
\LDa x(t_i)\approx \frac{1}{h^\a} \sum_{k=0}^{i}\w x(t_i-kh),
\end{equation}
where $\w=(-1)^k\binom{\a}{k}=\frac{\Gamma(k-\a)}{\Gamma(-\a)\Gamma(k+1)}$.

Similarly, one can approximate the right Riemann--Liouville  derivative by
\begin{equation}
\label{GLApprxR}
\RD x(t_i)\approx \frac{1}{h^\a} \sum_{k=0}^{n-i}\w x(t_i+kh).
\end{equation}
\end{remark}

\begin{remark}
The Gr\"{u}nwald--Letnikov approximation of Riemann--Liouville
is a first order approximation \cite{Podlubny}, i.e.,
\begin{equation*}
\LDa x(t_i)= \frac{1}{h^\a} \sum_{k=0}^{i}\w x(t_i-kh)+\mathcal{O}(h).
\end{equation*}
\end{remark}

\begin{remark}
It has been shown that the implicit Euler method solution to a certain fractional
partial differential equation based on the Gr\"{u}nwald--Letnikov approximation
to the fractional derivative, is unstable \cite{Meerschaert}. Therefore,
discretizing fractional derivatives,  shifted Gr\"{u}nwald--Letnikov derivatives
are used and, despite the slight difference, they exhibit a stable performance,
at least for certain cases. The shifted Gr\"{u}nwald--Letnikov derivative is defined by
\begin{equation*}
\sGLa x(t_i)\approx \frac{1}{h^\a} \sum_{k=0}^{i}\w x(t_i-(k-1)h).
\end{equation*}
\end{remark}

Other finite difference approximations can be found in the literature.
We refer here to the Diethelm backward finite difference formula for Caputo's
fractional derivative, with $0 < \a < 2$ and $\a \ne 1$,
which is an approximation of order $\mathcal{O}(h^{2-\a})$ \cite{Ford}:
$$
\LDC x(t_i)\approx \frac{h^{-\a}}{\Gamma(2-\a)}\sum_{j=0}^i a_{i,j}\left(x_{i-j}
-\sum_{k=0}^{\lfloor\a\rfloor}\frac{(i-j)^kh^k}{k!}x^{(k)}(a)\right),
$$
where
$$
a_{i,j}=\left \{
\begin{array}{ll}
1,& \text{if } i=0,\\
(j+1)^{1-\a}-2j^{1-\a}+(j-1)^{1-\a},&\text{if } 0<j<i,\\
(1-\a)i^{-\a}-i^{1-\a}+(i-1)^{1-\a},&\text{if } j=i.
\end{array}\right.
$$


\subsubsection{Euler-like Direct Method for Fractional Variational Problems}

As mentioned earlier, we consider a simple version of fractional variational
problems where the fractional term has a Riemann--Liouville form on a finite
time interval $[a,b]$. The boundary conditions are given and we approximate
the problem using the Gr\"{u}nwald--Letnikov approximation given by \eqref{GLApprx}.
In this context, we discretize the functional in \eqref{Functional} using
a simple quadrature rule on the mesh points, $a=t_0,t_1,, \ldots,t_n=b$,
with $h=\frac{b-a}{n}$. The goal is to find the values $x_1,x_2,\ldots,x_{n-1}$
of the unknown function $x(\cdot)$ at points $t_i$, $i=1,\ldots,n-1$.
The values of $x_0$ and $x_n$ are given. Applying the quadrature rule gives
\begin{eqnarray*}
J[x(\cdot)]&=& \sum_{i=1}^{n}\int_{t_{i-1}}^{t_{i}} L(t_i, x_i, \LDa x_i)dt
           \approx \sum_{i=1}^{n} h L(t_i, x_i, \LDa x_i)
\end{eqnarray*}
and by approximating the fractional derivatives
at mesh points using \eqref{GLApprx} we have
\begin{equation}
\label{disFuncl}
J[x(\cdot)]\approx \sum_{i=1}^{n}hL\left(t_i, x_i,
\frac{1}{h^\a} \sum_{k=0}^{i}\w x_{i-k}\right).
\end{equation}
Hereafter the procedure is the same as in the classical case.
The right-hand-side of \eqref{disFuncl} can be regarded as a
function $\Psi$ of $n-1$ unknowns $\mathbf{x}=(x_1,x_2,\ldots,x_{n-1})$,
\begin{equation}
\label{sumFrac}
\Psi(\mathbf{x})=\sum_{i=1}^{n}hL\left(t_i, x_i,
\frac{1}{h^\a} \sum_{k=0}^{i}\w x_{i-k}\right).
\end{equation}

To find an extremum for $\Psi$, one has to solve
the following system of algebraic equations:
\begin{equation}
\label{AlgSys}
\frac{\partial \Psi}{\partial x_i}=0,\qquad i=1,\ldots,n-1.
\end{equation}
Unlike the classical case, all terms, starting from the
$i$th term in \eqref{sumFrac}, depend on $x_i$ and we have
\begin{equation}
\label{GenSys}
\frac{\partial \Psi}{\partial x_i}
=h\frac{\partial L}{\partial x}(t_i,x_i,\LDa x_i)
+h\sum_{k=0}^{n-i}\frac{1}{h^\a}\w
\frac{\partial L}{\partial~\LDa x}(t_{i+k},x_{i+k},\LD x_{i+k}).
\end{equation}
Equating the right-hand-side of \eqref{GenSys} with zero, one has
\begin{equation*}
\frac{\partial L}{\partial x}(t_i,x_i,\LDa x_i)+\frac{1}{h^\a}
\sum_{k=0}^{n-i}\w\frac{\partial L}{\partial~\LDa x}(t_{i+k},x_{i+k},\LDa x_{i+k})=0.
\end{equation*}
Passing to the limit, and considering the approximation formula
for the right Riemann--Liouville derivative, equation \eqref{GLApprxR},
it is straightforward to verify that:

\begin{thm}
The Euler-like method for a fractional variational problem
of the form \eqref{Functional} is equivalent to the
fractional Euler--Lagrange equation
\begin{equation*}
\frac{\partial L}{\partial x}+\RD\frac{\partial L}{\partial~\LDa x}=0,
\end{equation*}
as the mesh size, $h$, tends to zero.
\end{thm}

\begin{proof}
Consider a minimizer $(x_1,\ldots,x_{n-1})$ of $\Psi$, a variation function
$\eta\in C[a,b]$ with $\eta(a)=\eta(b)=0$ and define $\eta_i=\eta(t_i)$,
for $i=0,\ldots,n$. We remark that $\eta_0=\eta_n=0$ and that
$(x_1+\epsilon\eta_1,\ldots,x_{n-1}+\epsilon\eta_{n-1})$ is a variation of
$(x_1,\ldots,x_{n-1})$, with $|\epsilon|<r$, for some fixed $r>0$. Therefore,
since $(x_1,\ldots,x_{n-1})$ is a minimizer for $\Psi$, proceeding
with Taylor's expansion, we deduce that
\begin{eqnarray*}
0&\leq& \Psi(x_1+\epsilon\eta_1,\ldots,x_{n-1}+\epsilon\eta_{n-1})-\Psi(x_1,\ldots,x_{n-1})\\
&=&\epsilon\sum_{i=1}^n h\left[ \frac{\partial L}{\partial x}[i]\eta_i
+\frac{\partial L}{\partial {_aD_t^{\a}}}[i] \frac{1}{h^\a}
\sum_{k=0}^i(\omega^\a_k) \eta_{i-k} \right]+\mathcal{O}(\epsilon),
\end{eqnarray*}
where
$$
[i]=\left(t_i,x_i,\frac{1}{h^\a}
\sum_{k=0}^i(\omega^\a_k) x_{i-k} \right).
$$
Since $\epsilon$ takes any value, it follows that
\begin{equation}
\label{sum1}
\sum_{i=1}^n h\left[ \frac{\partial L}{\partial x}[i]\eta_i
+\frac{\partial L}{\partial {_aD_t^{\a}}}[i] \frac{1}{h^\a}
\sum_{k=0}^i(\omega^\a_k) \eta_{i-k} \right]=0.
\end{equation}
On the other hand, since $\eta_0=0$, reordering
the terms of the sum, it follows immediately that
$$
\sum_{i=1}^n \frac{\partial L}{\partial {_aD_t^{\a}}}[i]
\sum_{k=0}^i(\omega^\a_k) \eta_{i-k}= \sum_{i=1}^n \eta_i
\sum_{k=0}^{n-i}(\omega^\a_k)  \frac{\partial L}{\partial {_aD_t^{\a}}}[i+k].
$$
Substituting this relation into equation \eqref{sum1}, we obtain
$$
\sum_{i=1}^n\eta_i h\left[ \frac{\partial L}{\partial x}[i]
+\frac{1}{h^\a} \sum_{k=0}^{n-i}(\omega^\a_k)
\frac{\partial L}{\partial {_aD_t^{\a}}}[i+k] \right]=0.
$$
Since $\eta_i$ is arbitrary, for $i=1,\ldots,n-1$, we deduce that
$$
\frac{\partial L}{\partial x}[i]+\frac{1}{h^\a} \sum_{k=0}^{n-i}(\omega^\a_k)
\frac{\partial L}{\partial {_aD_t^{\a}}}[i+k]=0, \quad \mbox{for } i=1,\ldots,n-1.
$$
Let us study the case when $n$ goes to infinity. Let $\overline t \in ]a,b[$
and $i\in\{1,\ldots,n\}$ such that $t_{i-1}<\overline t \leq t_i$. First observe
that, in such case, we also have $i\to\infty$ and $n-i\to\infty$. In fact,
let $i\in\{1,\ldots,n\}$ be such that
$$
a+(i-1)h<\overline t\leq a+ih.
$$
So, $i<(\overline t-a)/h+1$, which implies that
$$
n-i>n\frac{b-\overline t}{b-a}-1.
$$
Then
$$
\lim_{n\to\infty,i\to\infty}t_i=\overline t.
$$
Assume that there exists a function $\overline x\in C[a,b]$ satisfying
$$
\forall \epsilon>0\, \exists N \, \forall n\geq N \,: |x_i-\overline x(t_i)|
<\epsilon, \quad \forall i=1,\ldots,n-1.
$$
As $\overline x$ is uniformly continuous, we have
$$
\forall \epsilon>0\, \exists N \, \forall n\geq N \,: |x_i-\overline x(\overline t)|
<\epsilon, \quad \forall i=1,\ldots,n-1.
$$
By the continuity assumption of $\overline x$, we deduce that
$$
\lim_{n\to\infty,i\to\infty}\frac{1}{h^\a} \sum_{k=0}^{n-i}(\omega^\a_k)
\frac{\partial L}{\partial {_aD_t^{\a}}}[i+k]={_tD^{\a}_b }
\frac{\partial L}{\partial {_aD_t^{\a}}}(\overline t,
\overline x (\overline t),{_aD_{\overline t}^{\a}} \overline x(\overline t)).
$$
For $n$ sufficiently large (and therefore $i$ also sufficiently large),
$$
\lim_{n\to\infty,i\to\infty}\frac{\partial L}{\partial x}[i]
=\frac{\partial L}{\partial x}(\overline t, \overline x (\overline t),
{_aD_{\overline t}^{\a}} \overline x(\overline t)).
$$
In conclusion,
\begin{equation}
\label{fracELE}
\frac{\partial L}{\partial x}(\overline t, \overline x (\overline t),
{_aD_{\overline t}^{\a}} \overline x(\overline t))+{_tD^{\a}_b }
\frac{\partial L}{\partial {_aD_t^{\a}}}(\overline t,
\overline x (\overline t),{_aD_{\overline t}^{\a}} \overline x(\overline t))=0.
\end{equation}
Using the continuity condition, we prove that the fractional Euler--Lagrange
equation \eqref{fracELE} holds for all values on the closed interval $a\leq t\leq b$.
\end{proof}


\subsubsection{Examples}

Now we apply the Euler-like direct method to some test problems
for which the exact solutions are known. Although we propose
problems for the interval $[0,1]$, moving to arbitrary intervals
is only a matter of more computations. To measure the errors related
to approximations, different norms can be used. Since a direct method
seeks for the function values at certain points, we use the maximum norm
to determine how close we can get to the exact value at that point.
Assume that the exact value of the function $x(\cdot)$, at the point
$t_i$, is $x(t_i)$ and it is approximated by $x_i$. The error is defined as
$$
E=\max\{|x(t_i)-x_i|,~i=1,\cdots,n-1\}.
$$

\begin{example}
\label{Example1}
Our goal here is to minimize a quadratic Lagrangian on $[0,1]$
with fixed boundary conditions. Consider the following minimization problem:
\begin{equation}
\label{Exmp1}
\left\{
\begin{array}{l}
J[x(\cdot)]=\int_0^1 \left(\LDz x(t)-\frac{2}{\Gamma(2.5)}t^{1.5}\right)^2 dt \rightarrow \min\\
x(0)=0,~x(1)=1.
\end{array}
\right.
\end{equation}
Since the Lagrangian is always positive,
problem \eqref{Exmp1} attains its minimum when
$$
\LDz x(t)-\frac{2}{\Gamma(2.5)}t^{1.5}=0
$$
and has the obvious solution of the form $x(t)=t^2$
because $\LDz t^2=\frac{2}{\Gamma(2.5)}t^{1.5}$.
\end{example}

To begin with, we approximate the fractional derivative by
\begin{equation*}
\LDz x(t_i)\approx \frac{1}{h^{0.5}} \sum_{k=0}^{i}\wh x(t_i-kh)
\end{equation*}
for a fixed $h>0$. The functional is now transformed into
\begin{equation*}
J[x(\cdot)]=\int_0^1 \left(\frac{1}{h^{0.5}}
\sum_{k=0}^{i}\wh x_{i-k}-\frac{2}{\Gamma(2.5)}t^{1.5}\right)^2dt.
\end{equation*}
Finally, we approximate the integral by a rectangular rule
and end with the discrete problem
\begin{equation*}
\Psi(\mathbf{x})=\sum_{i=1}^{n}h\left(\frac{1}{h^{0.5}}
\sum_{k=0}^{i}\wh x_{i-k}-\frac{2}{\Gamma(2.5)}t^{1.5}_i\right)^2.
\end{equation*}
Since the Lagrangian in this example is quadratic, system \eqref{AlgSys}
has a linear form and therefore is easy to solve. Other problems may end
with a system of nonlinear equations. Simple calculations lead to the system
\begin{equation}
\label{Ex1Sys}
\mathbf{A}\mathbf{x}=\mathbf{b},
\end{equation}
in which
\begin{equation*}
\mathbf{A}=\left[\begin{array}{llll}
\sum_{i=0}^{n-1}A_i^2        & \sum_{i=1}^{n-1}A_{i}A_{i-1}
& \cdots & \sum_{i=n-2}^{n-1}A_{i}A_{i-(n-2)} \\
\sum_{i=0}^{n-2}A_{i}A_{i+1} & \sum_{i=1}^{n-2}A_i^2
& \cdots & \sum_{i=n-3}^{n-2}A_{i}A_{i-(n-3)} \\
\sum_{i=0}^{n-3}A_{i}A_{i+2} & \sum_{i=1}^{n-3}A_{i}A_{i+1}
& \cdots & \sum_{i=n-4}^{n-3}A_{i}A_{i-(n-4)} \\
\vdots                       & \vdots
& \ddots & \vdots                        \\
\sum_{i=0}^{1}A_{i}A_{i+n-2} & \sum_{i=0}^{1}A_{i}A_{i+n-3}
& \cdots & \sum_{i=0}^{1}A_i^2
\end{array}\right],
\end{equation*}
where $A_i=(-1)^{i}h^{1.5}\binom{0.5}{i}$
and $\mathbf{b}=(b_1,b_2,\cdots,b_{n-1})$ with
\begin{equation*}
b_i=\sum_{k=0}^{n-i}\frac{2h^2A_k}{\Gamma(2.5)}t_{k+i}^{1.5}
-A_{n-i}A_0-\left(\sum_{k=0}^{n-i}A_kA_{k+i}\right).
\end{equation*}
Since system \eqref{Ex1Sys} is linear, it is easily solved for different
values of $n$. As indicated in Figure~\ref{Ex1Fig}, by increasing
the value of $n$ we get better solutions.
\begin{figure}[!tp]
\begin{center}
\includegraphics[scale=.8]{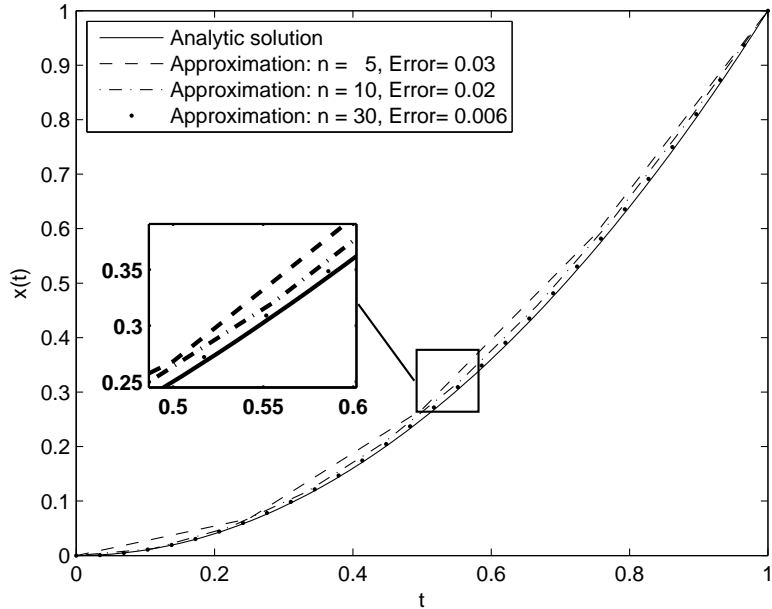}
\end{center}
\caption{Analytic and approximate solutions of Example~\ref{Example1}.}\label{Ex1Fig}
\end{figure}

Let us now move to another example for which the solution
is obtained by the fractional Euler--Lagrange equation.

\begin{example}
\label{Example2}
Consider the following minimization problem:
\begin{equation}
\label{Exmp2}
\left\{
\begin{array}{l}
J[x(\cdot)]=\int_0^1 \left(\LDz x(t)-\dot{x}^2(t)\right) dt \rightarrow \min\\
x(0)=0,~x(1)=1.
\end{array}
\right.
\end{equation}
In this case the only way to get a solution is by use of Euler--Lagrange equations.
The Lagrangian depends not only on the fractional derivative, but also on the first
order derivative of the function. The Euler--Lagrange equation for this setting becomes
\begin{equation*}
\frac{\partial L}{\partial x}+\RD \frac{\partial L}{\partial \LDa}
-\frac{d}{dt}\left(\frac{\partial L}{\partial \dot{x}}\right)=0,
\end{equation*}
and by direct computations a necessary condition for $x(\cdot)$
to be a minimizer of \eqref{Exmp2} is
$$
_tD_1^{\a} 1+2\ddot{x}(t)=0~\text{ or }~ \ddot{x}(t)
=\frac{1}{2\Gamma(1-\a)}(1-t)^{-\a}.
$$
Subject to the given boundary conditions, the above second order
ordinary differential equation has the solution
\begin{equation}
\label{solEx51}
x(t)=-\frac{1}{2\Gamma(3-\a)}(1-t)^{2-\a}+\left(1
-\frac{1}{2\Gamma(3-\a)}\right)t+\frac{1}{2\Gamma(3-\a)}.
\end{equation}
\end{example}

Discretizing problem \eqref{Exmp2} with the same assumptions
of Example~\ref{Example1} ends in a linear system of the form
\begin{equation}
\label{Ex2Sys}
\left[\begin{array}{ccccccc}
2      & -1    & 0     & 0    & \cdots & 0 & 0\\
-1     & 2     & -1    & 0    & \cdots & 0 & 0\\
0      & -1    & 2     & -1   & \cdots & 0 & 0\\
\vdots & \vdots& \vdots&\vdots& \ddots & \vdots&\vdots \\
0      &  0    & 0     &  0   & \cdots &-1 & 2
\end{array}\right]\left[\begin{array}{c}x_1\\x_2\\x_3\\\vdots\\x_{n-1}\end{array}\right]=
\left[\begin{array}{c}b_1\\b_2\\b_3\\\vdots\\b_{n-1}\end{array}\right],
\end{equation}
where
$$
b_i=\frac{h}{2}\sum_{k=0}^{n-i-1}(-1)^{k}h^{0.5}\binom{0.5}{k},
\qquad i=1,2,\ldots,n-2,
$$
and
$$
b_{n-1}=\frac{h}{2}\sum_{k=0}^{1}\left((-1)^{k}h^{0.5}\binom{0.5}{k}\right)+x_n.
$$
System \eqref{Ex2Sys} is linear and can be solved for any $n$ to reach the desired
accuracy. The analytic solution together with some approximated solutions
are shown in Figure~\ref{Ex2Fig}.
\begin{figure}[!tp]
\begin{center}
\includegraphics[scale=.8]{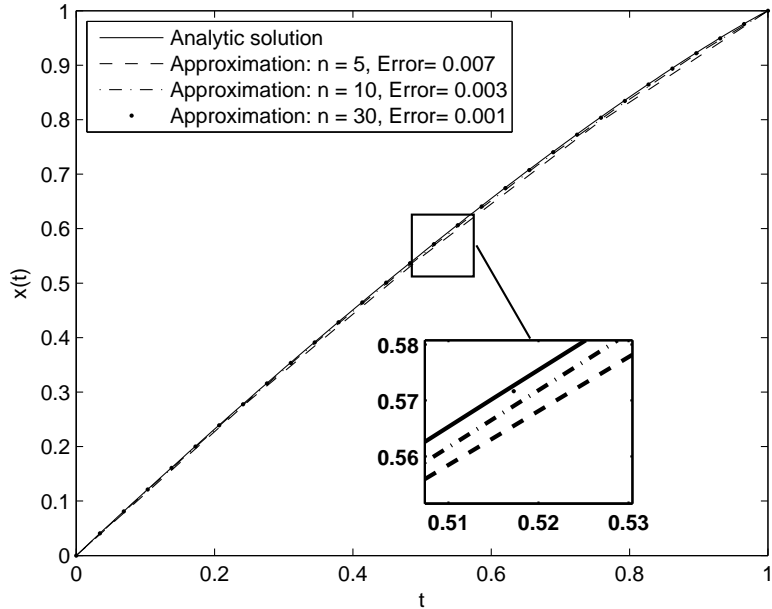}
\end{center}
\caption{Analytic and approximate solutions of Example~\ref{Example2}.}\label{Ex2Fig}
\end{figure}

Both examples above end with linear systems and their solvability is simply dependant
to the matrix of coefficients. Now we try this method on a more complicated problem,
yet analytically solvable, with an oscillating solution.

\begin{example}
\label{Example3}
Consider the problem of minimizing $\int_0^1 L dt$ subject
to the boundary conditions $x(0) = 0$ and $x(1) = 1$,
where the Lagrangian $L$ is given by
\begin{equation*}
L=\left(\LDz x(t)-\frac{16\Gamma(6)}{\Gamma(5.5)}t^{4.5}
+\frac{20\Gamma(4)}{\Gamma(3.5)}t^{2.5}
-\frac{5}{\Gamma(1.5)}t^{0.5}\right)^4.
\end{equation*}
This example has an obvious solution too.
Since $L$ is positive, the minimizer is
\begin{equation*}
x(t)=16t^{5}-20t^{3}+5t.
\end{equation*}
Note that $\LD (t-a)^\nu=\frac{\Gamma(\nu+1)}{\Gamma(\nu+\-\a)}t^{\nu-\a}$.
\end{example}

The appearance of a fourth power in the Lagrangian, results in a nonlinear system
as we apply the Euler-like direct method to this problem. For $j=1,\cdots,n-1$ we have
\begin{equation}
\label{Sys3}
\sum_{i=j}^n \left(\omega_{i-j}^{0.5}\right)\left(\frac{1}{h^{0.5}}
\sum_{k=0}^{i}\wh x_{i-k}-\phi(t_i)\right)^3=0,
\end{equation}
where
$$
\phi(t)=\frac{16\Gamma(6)}{\Gamma(5.5)}t^{4.5}
+\frac{20\Gamma(4)}{\Gamma(3.5)}t^{2.5}-\frac{5}{\Gamma(1.5)}t^{0.5}.
$$
System \eqref{Sys3} is solved for different values
of $n$ and the results are depicted in Figure~\ref{Ex3Fig}.
\begin{figure}[!htp]
\begin{center}
\includegraphics[scale=.8]{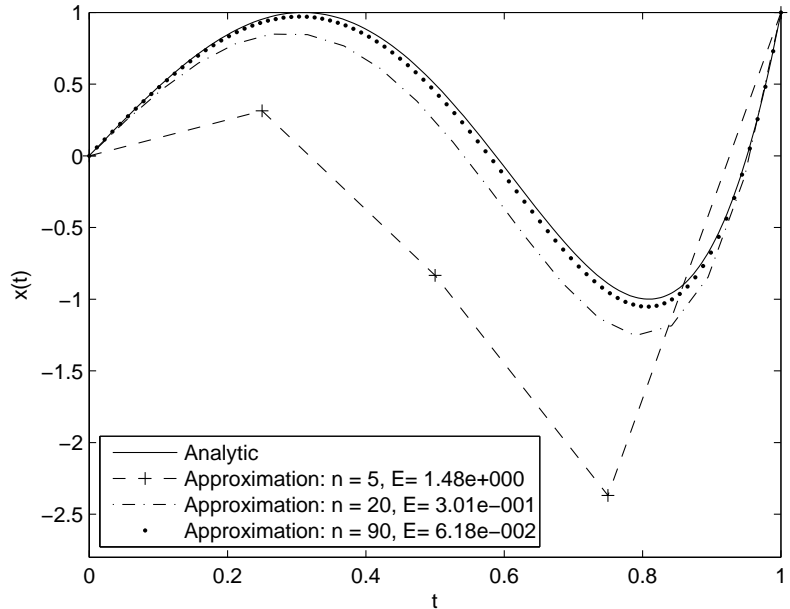}
\end{center}
\caption{Analytic and approximate solutions of Example~\ref{Example3}.}\label{Ex3Fig}
\end{figure}


\section{Indirect Methods}

As in the classical case, indirect methods in fractional sense provide
the necessary conditions of optimality using the first variation.
Fractional Euler--Lagrange equations are now a well-known and well-studied
subject in fractional calculus. For a simple problem of the form \eqref{Functional},
following \cite{Agrawal}, a necessary condition implies that
the solution must satisfy a fractional boundary value differential equation.

\begin{thm}[\bf cf. \cite{Agrawal}]
Let $x(\cdot)$ have a continuous left Riemann--Liouville
derivative of order $\a$ and $J$ be a functional of the form
\begin{equation}
\label{CopyMainProb}
J[x(\cdot)]=\int_a^b L(t, x(t), \LD x(t))dt
\end{equation}
subject to the boundary conditions $x(a)=x_a$ and $x(b)=x_b$.
Then a necessary condition for $J$ to have an extremum for a
function $x(\cdot)$ is that $x(\cdot)$ satisfies
the following Euler–-Lagrange equation:
\begin{equation}
\label{ELeqnIndirect}
\left\{\begin{array}{l}
\frac{\partial L}{\partial x}+\RD\frac{\partial L}{\partial~\LDa x}=0,\\
x(a)=x_a,\quad x(b)=x_b,
\end{array}\right.
\end{equation}
which is called the fractional Euler--Lagrange equation.
\end{thm}
\begin{proof}
Assume that $x^*(t)$ is the desired function and let $x(t)=x^*(t)+\epsilon \eta(t)$
be a family of curves that satisfy boundary conditions, i.e., $\eta(a)=\eta(b)=0$.
Since $\LD$ is a linear operator, for any $x(\cdot)$, the functional becomes
$$
J[x(\cdot)]=\int_a^b L(t, x^*(t)+\epsilon \eta(t), \LD x^*(t)+\epsilon \LD\eta(t))dt,
$$
which is a function of $\epsilon$, $J[\epsilon]$. Since $J$ assumes its extremum
at $\epsilon=0$, one has $\frac{dJ}{d\epsilon}\big |_{\epsilon=0}=0$, i.e.,
$$
\int_a^b\left[\frac{\partial L}{\partial x}\eta
+ \frac{\partial L}{\partial \LD x}\LD\eta\right]dt=0.
$$
Using the fractional integration by parts of the form
$$
\int_a^b g(t)\LD f(t)dt=\int_a^b f(t)\RD g(t)dt
$$
on the second term and applying the fundamental theorem
of the calculus of variations completes the proof.
\end{proof}

\begin{remark}
Many variants of this theorem can be found in the literature.
Different types of fractional terms have been embedded in the
Lagrangian and appropriate versions of Euler--Lagrange equations
have been derived using proper integration by parts formulas.
See \cite{Agrawal,APTIndInt,Atan,Malinowska,ID:207} for details.
\end{remark}

For fractional optimal control problems, a so-called Hamiltonian system
is constructed using Lagrange multipliers. For example, cf. \cite{Ozlem},
assume that we are required  to minimize a functional of the form
$$
J[x(\cdot),u(\cdot)]=\int_a^b L(t, x(t),u(t))dt
$$
such that $x(a)=x_a$, $x(b)=x_b$ and $\LD x(t)=f(t, x(t),u(t))$.
Similar to the classical methods, one can introduce a Hamiltonian
$$
H=L(t, x(t),u(t))+\lambda(t) f(t, x(t),u(t)),
$$
where $\lambda(t)$ is considered as a Lagrange multiplier.
In this case we define the augmented functional as
$$
J[x(\cdot), u(\cdot), \lambda(\cdot)]
=\int_a^b [H(t, x(t),u(t),\lambda(t))-\lambda(t)\LD x(t)]dt.
$$
Optimizing the latter functional results
in the following necessary optimality conditions:
\begin{equation}
\label{HamIndirect}
\left\{\begin{array}{l}
\LD x(t)=\frac{\partial H}{\partial \lambda}\\
\RD \lambda(t)=\frac{\partial H}{\partial x}\\
\frac{\partial H}{\partial u}=0.
\end{array}\right.
\end{equation}
Together with the prescribed boundary conditions,
this makes a two point fractional boundary value problem.

These arguments reveal that, like the classical case, fractional variational problems
end with fractional boundary value problems. To reach an optimal solution, one needs
to deal with a fractional differential equation or a system of fractional differential equations.

The classical theory of differential equations is furnished with several solution methods,
theoretical and numerical. Nevertheless, solving a fractional differential equation
is a rather tough task \cite{Kai}. To benefit those methods, especially all solvers
that are available to solve an integer order differential equation numerically,
we can either approximate a fractional variational problem by an equivalent
integer-order one or approximate the necessary optimality conditions
\eqref{ELeqnIndirect} and \eqref{HamIndirect}. The rest of this section discusses
two types of approximations that are used to transform a fractional problem
to one in which only integer order derivatives are present; i.e., we approximate
the original problem by substituting a fractional term by its corresponding
expansion formulas. This is mainly done by case studies on certain examples.
The examples are chosen so that either they have a trivial solution or it is
possible to get an analytic solution using fractional Euler--Lagrange equations.

By substituting the approximations \eqref{expanInt} or \eqref{expanMom}
for the fractional derivative in \eqref{CopyMainProb}, the problem is transformed to
\begin{eqnarray*}
J[x(\cdot)]&=&\int_a^b L\left(t, x(t), \sum_{k=0}^{N}
\frac{(-1)^{k-1}\a x^{(k)}(t)}{k!(k-\a)\Gamma(1-\a)}(t-a)^{k-\a}\right)dt\\
           &=&\int_a^b L'\left(t, x(t), \dot{x}(t), \ldots,x^{(N)}(t)\right)dt
\end{eqnarray*}
or
\begin{eqnarray*}
J[x(\cdot)]&=&\int_a^b L\left(t, x(t),\frac{Ax(t)}{(t-a)^{\a}}
+\frac{B\dot{x}(t)}{(t-a)^{\a-1}}-\sum_{p=2}^N \frac{C(\a,p)V_p(t)}{(t-a)^{p+\a-1}}\right)dt\\
           &=&\int_a^b L'\left(t, x(t), \dot{x}(t), V_2(t), \ldots,V_N(t)\right)dt\\
           &&\left\{
           \begin{array}{l}
           \dot{V}_p(t)=(1-p)(t-a)^{p-2}x(t)\\
           V_p(a)=0, \qquad p=2,3,\ldots
           \end{array}
           \right.
\end{eqnarray*}
The former problem is a classical variational problem containing higher order derivatives.
The latter is a multi-variable problem, subject to some ordinary differential equation constraint.
Together with the boundary conditions, both above problems belong to classes of well studied variational problems.

To accomplish a detailed study, as test problems, we consider here Example~\ref{Example2},
\begin{equation}
\label{Exmp1Indirect}
\left\{
\begin{array}{l}
J[x(\cdot)]=\int_0^1 \left(\LDz x(t)-\dot{x}^2(t)\right) dt \rightarrow \min\\
x(0)=0,~x(1)=1,
\end{array}
\right.
\end{equation}
and the following example.
\begin{example}
\label{Example4}
Given $\a\in (0,1)$, consider the functional
\begin{equation}
\label{Exmp2Indirect}
J[x(\cdot)]=\int_0^1 (\LD x(t)-1)^2dt
\end{equation}
to be minimized subject to the boundary conditions $x(0)=0$ and $x(1)=\frac{1}{\Gamma(\a+1)}$.
Since the integrand in \eqref{Exmp2Indirect} is non-negative, the functional attains its minimum
when $\LD x(t)=1$, \textrm{i.e.}, for $x(t)=\frac{t^{\a}}{\Gamma(\a+1)}$.
\end{example}

We illustrate the use of the two different expansions separately.


\subsection{Expansion to Integer Orders}

Using approximation \eqref{expanInt} for the fractional derivative
in \eqref{Exmp1Indirect}, we get the approximated problem
\begin{equation}
\label{expanCOV}
\begin{aligned}
\min\quad &\tilde{J}[x(\cdot)]=\int_0^1 \left[\sum_{n=0}^N
C(n,\a)t^{n-\a}x^{(n)}(t)-\dot{x}^2(t)\right]dt\\
&x(0)=0,\quad x(1)=1,
\end{aligned}
\end{equation}
which is a classical higher-order problem of the calculus of variations
that depends on derivatives up to order $N$.
The corresponding necessary optimality condition is a well-known result.

\begin{thm}[\bf \textrm{cf.}, \textrm{e.g.}, \cite{Lebedev}]
Suppose that $x(\cdot)\in C^{2N}[a,b]$ minimizes
$$
\int_a^b L(t,x(t),x^{(1)}(t),x^{(2)}(t),\ldots,x^{(N)}(t))dt
$$
with given boundary conditions
\begin{eqnarray*}
x(a)=a_0, & &x(b)=b_0,\\
x^{(1)}(a)=a_1, & &x^{(1)}(b)=b_1,\\
&\vdots&\\
x^{(N-1)}(a)=a_{N-1}, & & x^{(N-1)}(b)=b_{N-1}.
\end{eqnarray*}
Then $x(\cdot)$ satisfies the Euler--Lagrange equation
\begin{equation}
\label{ELN}
\frac{\partial L}{\partial x}-\frac{d}{dt}\left(\frac{\partial L}{\partial x^{(1)}}\right)
+\frac{d^2}{dt^2}\left(\frac{\partial L}{\partial x^{(2)}}\right)
-\cdots+(-1)^N\frac{d^N}{dt^N}\left(\frac{\partial L}{\partial x^{(N)}}\right)=0.
\end{equation}
\end{thm}

In general \eqref{ELN} is an ODE of order $2N$, depending on the order $N$
of the approximation we choose, and the method leaves $2N-2$ parameters unknown.
In our example, however, the Lagrangian in \eqref{expanCOV} is linear with respect
to all derivatives of order higher than two.
The resulting Euler--Lagrange equation is the second order ODE
\begin{equation*}
\sum_{n=0}^N (-1)^nC(n,\a)\frac{d^n}{dt^n}(t^{n-\a})-\frac{d}{dt}\left[-2\dot{x}(t)\right]=0
\end{equation*}
that has the solution
\begin{equation*}
x(t)=M_1(\a,N)t^{2-\a}+M_2(\a,N)t,
\end{equation*}
where
\begin{eqnarray*}
M_1(\a,N)&=&-\frac{1}{2\Gamma(3-\a)}\left[\sum_{n=0}^N(-1)^n\Gamma(n+1-\a)C(n,\a)\right],\\
M_2(\a,N)&=&\left[1+\frac{1}{2\Gamma(3-\a)}\sum_{n=0}^N(-1)^n\Gamma(n+1-\a)C(n,\a)\right].
\end{eqnarray*}
Figure~\ref{expIntFig} shows the analytic solution together with several approximations.
It reveals that by increasing $N$, approximate solutions do not converge to the analytic one.
The reason is the fact that the solution \eqref{solEx51} to Example~\ref{Example2} is not
an analytic function. We conclude that \eqref{expanInt} may not be a good choice
to approximate fractional variational problems. In contrast, as we shall see,
the approximation \eqref{expanMom} leads to good results.
\begin{figure}
\begin{center}
\includegraphics[scale=.6]{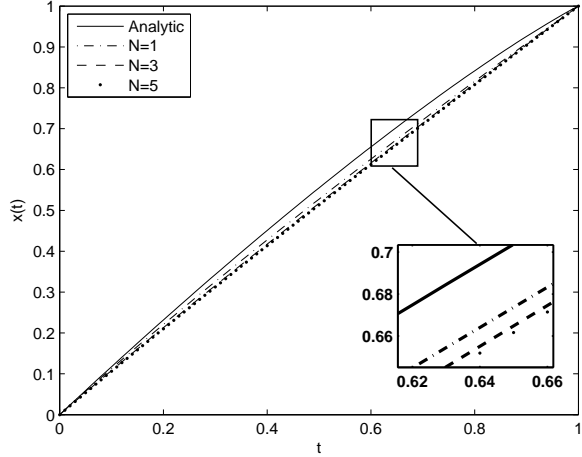}
\caption{Analytic versus approximate solutions to Example~\ref{Example2}
using approximation \eqref{expanInt} with $\a=0.5$.}
\label{expIntFig}
\end{center}
\end{figure}

To solve Example~\ref{Example2} using \eqref{expanInt} as an approximation
for the fractional derivative, the problem becomes
\begin{equation*}
\begin{aligned}
\min\quad &\tilde{J}[x(\cdot)]=\int_0^1 \left(\sum_{n=0}^N
C(n,\a)t^{n-\a}x^{(n)}(t)-1\right)^2dt,\\
&x(0)=0,\quad x(1)=\frac{1}{\Gamma(\a+1)}.
\end{aligned}
\end{equation*}
The Euler--Lagrange equation \eqref{ELN} gives a $2N$ order ODE.
For $N \ge 2$ this approach is inappropriate since the two given
boundary conditions $x(0)=0$ and $x(1)=\frac{1}{\Gamma(\a+1)}$
are not enough to determine the $2N$ constants of integration.


\subsection{Expansion through the Moments of a Function}

If we use \eqref{expanMom} to approximate the optimization
problem \eqref{Exmp1Indirect}, with $A=A(\a,N)$, $B=B(\a,N)$
and $C_p=C(\a,p)$, we have
\begin{equation}
\label{momCOV}
\begin{split}
\tilde{J}[x(\cdot)]&=\int_0^1 \left[At^{-\a}x(t)
+Bt^{1-\a}\dot{x}(t)-\sum_{p=2}^N
C_pt^{1-p-\a}V_p(t)-\dot{x}^2(t)\right]dt,\\
\dot{V}_p(t)&=(1-p)t^{p-2}x(t),\quad p=2,\ldots,N,\\
V_p(0)&=0, \quad p=2,\ldots,N,\\
x(0)&=0,\quad x(1)=1.
\end{split}
\end{equation}
Problem \eqref{momCOV} is constrained with a set of ordinary differential
equations and is natural to look to it as an optimal control problem \cite{Pontryagin}.
For that we introduce the control variable $u(t) = \dot{x}(t)$. Then, using
the Lagrange multipliers $\lambda_1, \lambda_2,\ldots,\lambda_N$,
and the Hamiltonian system, one can reduce \eqref{momCOV} to the study
of the two point boundary value problem
\begin{equation}
\label{tpbvp}
\left\{
\begin{array}{rl}
\dot{x}(t)&=\frac{1}{2}Bt^{1-\a}-\frac{1}{2}\lambda_1(t),\\
\dot{V}_p(t)&=(1-p)t^{p-2}x(t),\quad p=2,\ldots,N,\\
\dot{\lambda}_1(t)&=At^{-\a}-\sum_{p=2}^N(1-p)t^{p-2}\lambda_p(t),\\
\dot{\lambda}_p(t)&=-C_pt^{(1-p-\a)},\quad p=2,\ldots,N,\\
\end{array}\right.
\end{equation}
with boundary conditions
\begin{equation*}
\left\{
\begin{array}{l}
x(0)=0, \\
V_p(0)=0,\quad p=2,\ldots,N,
\end{array}
\right.\qquad\left\{
\begin{array}{l}
x(1)=1,\\
\lambda_p(1)=0,\quad p=2,\ldots,N,
\end{array}
\right.
\end{equation*}
where $x(0)=0$ and $x(1)=1$ are given. We have $V_p(0)=0$, $p=2,\ldots,N$,
due to \eqref{sysVp} and $\lambda_p(1)=0$, $p=2,\ldots,N$, because $V_p$
is free at final time for $p=2,\ldots,N$ \cite{Pontryagin}. In general,
the Hamiltonian system is a nonlinear, hard to solve, two point boundary
value problem that needs special numerical methods. In this case, however,
\eqref{tpbvp} is a non-coupled system of ordinary differential equations
and is easily solved to give
\begin{equation*}
x(t)=M(\a,N)t^{2-\a}-\sum_{p=2}^N \frac{C(\a,p)}{2p(2-p-\a)}t^p
+\left[ 1-M(\a,N)+\sum_{p=2}^N \frac{C(\a,p)}{2p(2-p-\a)}\right]t,
\end{equation*}
where
$$
M(\a,N)=\frac{1}{2(2-\a)}\left[ B(\a,N)-\frac{A(\a,N)}{1-\a}
-\sum_{p=2}^N \frac{C(\a,p)(1-p)}{(1-\a)(2-p-\a)} \right].
$$
Figure~\ref{expMomFig} shows the graph of $x(\cdot)$ for different values of $N$.
\begin{figure}
\begin{center}
\includegraphics[scale=.6]{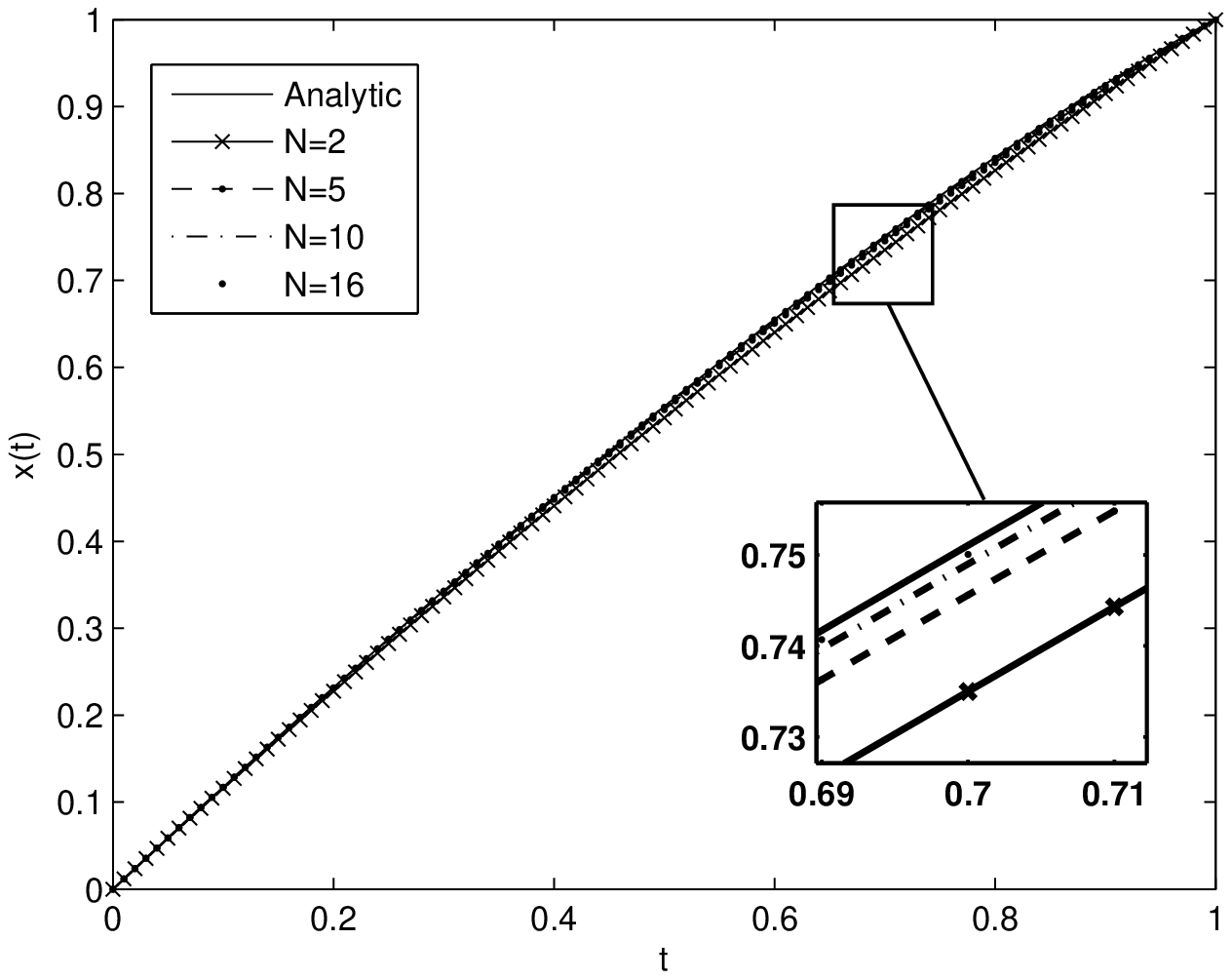}
\caption{Analytic versus approximate solutions to Example~\ref{Example2}
using approximation \eqref{expanMom} with $\a=0.5$.}
\label{expMomFig}
\end{center}
\end{figure}

Let us now approximate Example~\ref{Example4} using \eqref{expanMom}.
The resulting minimization problem has the following form:
\begin{equation}
\label{ex52Mom}
\begin{aligned}
\min\quad &\tilde{J}[x(\cdot)]=\int_0^1 \left[At^{-\a}x(t)
+Bt^{1-\a}\dot{x}(t)-\sum_{p=2}^N
C_pt^{1-p-\a}V_p(t)-1\right]^2dt,\\
& \dot{V}_p(t)=(1-p)t^{p-2}x(t),\quad p=2,\ldots,N,\\
& V_p(0)=0, \quad p=2,\ldots,N,\\
&x(0)=0,\quad x(1)=\frac{1}{\Gamma(\a+1)}.
\end{aligned}
\end{equation}
Following the classical optimal control approach
of Pontryagin \cite{Pontryagin}, this time with
$$
u(t)=At^{-\a}x(t)+Bt^{1-\a}\dot{x}(t)
-\sum_{p=2}^N C_pt^{1-p-\a}V_p(t),
$$
we conclude that the solution to \eqref{ex52Mom}
satisfies the system of differential equations
\begin{equation}
\label{ex52tpbvp}
\left\{
\begin{array}{rl}
\dot{x}(t)&=-AB^{-1}t^{-1}x(t)+\sum_{p=2}^NB^{-1}C_pt^{-p}V_p(t)
+\frac{1}{2}B^{-2}t^{2\a-2}\lambda_1(t)+B^{-1}t^{\a-1},\\
\dot{V}_p(t)&=(1-p)t^{p-2}x(t),\quad p=2,\ldots,N,\\
\dot{\lambda}_1(t)&=AB^{-1}t^{-1}\lambda_1
-\sum_{p=2}^N(1-p)t^{p-2}\lambda_p(t),\\
\dot{\lambda}_p(t)&=-B^{-1}C_pt^{-p}\lambda_1,
\quad p=2,\ldots,N,\\
\end{array}\right.
\end{equation}
where $A=A(\a,N)$, $B=B(\a,N)$ and $C_p=C(\a,p)$ are defined
according to Section~\ref{secexpan},
subject to the boundary conditions
\begin{equation}
\label{sysB52}
\left\{
\begin{array}{l}
x(0)=0,\\
V_p(0)=0,\quad p=2,\ldots,N,
\end{array}
\right.\qquad\left\{
\begin{array}{l}
x(1)=\frac{1}{\Gamma(\a+1)},\\
\lambda_p(1)=0,\quad p=2,\ldots,N.
\end{array}
\right.
\end{equation}
The solution to system \eqref{ex52tpbvp}--\eqref{sysB52},
with $N=2$, is shown in Figure~\ref{expMomFig52}.
\begin{figure}
\begin{center}
\includegraphics[scale=.6]{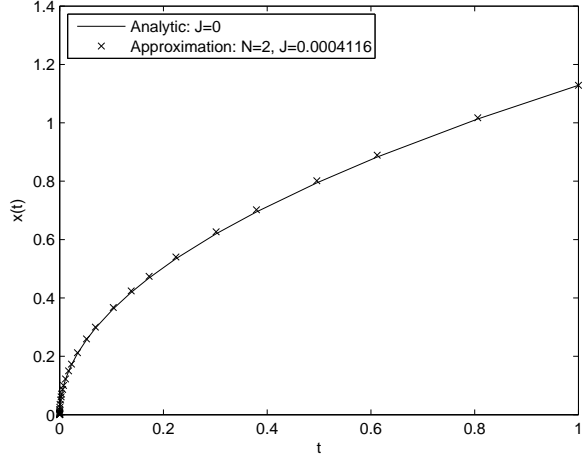}
\caption{Analytic versus approximate solution to Example~\ref{Example4}
using approximation \eqref{expanMom} with $\a=0.5$.}
\label{expMomFig52}
\end{center}
\end{figure}


\section{Conclusion}

The realm of numerical methods in scientific fields is vastly growing due
to the very fast progresses in computational sciences and technologies.
Nevertheless, the intrinsic complexity of fractional calculus, caused
partially by non-local properties of fractional derivatives and integrals,
makes it rather difficult to find efficient numerical methods in this field.
It seems enough to mention here that, up to the time of this manuscript,
and to the best of our knowledge, there is no routine available for solving
a fractional differential equation as Runge--Kutta for ordinary ones. Despite
this fact, however, the literature exhibits a growing interest and improving
achievements in numerical methods for fractional calculus in general
and fractional variational problems specifically.

This chapter was devoted to discuss some aspects of the very well-known methods
for solving variational problems. Namely, we studied the notions of direct
and indirect methods in the classical calculus of variations and we also mentioned
some connections to optimal control. Consequently, we introduced the generalizations
of these notions to the field of fractional calculus of variations and fractional optimal control.

The method of finite differences, as discussed here, seems to be a potential first
candidate to solve fractional variational problems. Although a first order approximation
was used for all examples, the results are satisfactory and even though it is more complicated
than in the classical case, it still inherits some sort of simplicity and an ease of implementation.

The lack of efficient numerical methods for fractional variational problems is overcome,
partially, by the indirect methods of this chapter. Once we transformed the fractional
variational problem to an approximated classical one, the majority of classical methods
can be applied to get an approximate solution. Nevertheless, the procedure
is not completely straightforward. The singularity of fractional operators is still present
in the approximating formulas and it makes the solution procedure more complicated.


\section*{Acknowledgements}

Part of first author's Ph.D., carried out
at the University of Aveiro under the Doctoral Program
in Mathematics and Applications (PDMA)
of Universities of Aveiro and Minho.
Work supported by {\it FEDER} funds through {\it COMPETE}
--- Operational Programme Factors of Competitiveness
(``Programa Operacional Factores de Competitividade'')
and by Portuguese funds through the {\it Center for Research
and Development in Mathematics and Applications} (University of Aveiro)
and the Portuguese Foundation for Science and Technology
(``FCT--Funda\c{c}\~{a}o para a Ci\^{e}ncia e a Tecnologia''),
within project PEst-C/MAT/UI4106/20\-11 with COMPETE
number FCOMP-01-0124-FEDER-022690. Pooseh was also
supported by the FCT Ph.D. fellowship SFRH/BD/33761/2009;
Torres by EU funding under the 7th Framework Programme FP7-PEOPLE-2010-ITN,
grant agreement no.~264735-SADCO.


\bigskip


\end{document}